\documentclass[amssymb,reqno,amsfonts,refcheck,12pt,verbatim,righttag]{amsart}

\usepackage{graphicx}
\usepackage{graphics,color}
\usepackage{amssymb,mathrsfs}
\usepackage{graphicx,color,tikz,caption,subcaption}
\usepackage[all]{xy}
\usepackage{mathtools}
\usepackage{enumerate}
\usepackage{amsmath}
\usepackage{bbm}
\usepackage[colorlinks, linkcolor=blue,citecolor=blue, anchorcolor=blue,urlcolor=blue,hypertexnames=false]{hyperref}

\numberwithin{equation}{section} 
\newtheorem{thm}{Theorem}[section]
\newtheorem{lemma}[thm]{Lemma}

\newtheorem{cor}[thm]{Corollary}

\newtheorem{de}[thm]{Definition}
\newtheorem{rem}[thm]{Remark}

\newcommand{\R}{\mathbb{R}}

\newcommand{\N}{\mathbb{N}}

\def\d{ \,\mathrm{d}}

\usepackage{float}
\begin{document}
\title{Quantitative Diophantine approximation and Fourier dimension of sets$\colon$
Dirichlet non-improvable numbers versus well-approximable numbers}
\author{Bo Tan \and Qing-Long Zhou$^{\dag}$}
\address{$^{1}$ School  of  Mathematics  and  Statistics,
                Huazhong  University  of Science  and  Technology, 430074 Wuhan, PR China}
          \email{tanbo@hust.edu.cn}

\address{$^{2}$ Department of Mathematics, Wuhan University of Technology, 430070 Wuhan, PR China }
\email{zhouql@whut.edu.cn}
\thanks{$^{\dag}$ Corresponding author.}

\date{}

\begin{abstract}
Let $E\subset [0,1]$ be a set that supports a probability measure $\mu$ with the property that $|\widehat{\mu}(t)|\ll (\log |t|)^{-A}$ for some constant $A>2.$	
Let $\mathcal{A}=(q_n)_{n\in \N}$ be a positive, real-valued, lacunary sequence.
We  present a quantitative inhomogeneous Khintchine-type theorem in
which the points of interest are restricted to $E$ and the denominators of the shifted fractions are restricted to $\mathcal{A}.$
Our result improves and extends a previous result in this direction obtained by Pollington-Velani-Zafeiropoulos-Zorin (2022). We also show that the Dirichlet non-improvable set VS well-approximable set is of positive Fourier dimension.

\end{abstract}
\keywords{Equidistribution, Inhomogeneous Diophantine approximation, Fourier dimension.}
\subjclass[2010]{Primary 28A80; Secondary 11K55, 11J83}

\maketitle

\section{Introduction}
\subsection{Metric Diophantine approximation}
The classical metric Diophantine approximation is concerned with the rational approximations to real numbers.
A qualitative answer is provided by the fact that the rationals are dense in the set of reals.
Dirichlet's theorem initiates the quantitative description of the rational approximation, which states that for any $x\in \mathbb{R}$ and $Q>1,$
there exists $(p,q)\in\mathbb{Z}\times\mathbb{N}$ such that
\begin{equation}\label{D1}
|qx-p|\le \frac{1}{Q}~~\text{ and } ~~q<Q.
\end{equation}
A direct corollary of  (\ref{D1}) reads$\colon$ for any $x\in \mathbb{R},$
there exists infinitely many $(p,q)\in\mathbb{Z}\times\mathbb{N}$ such that
\begin{equation}\label{D2}
|qx-p|\le \frac{1}{q}.
\end{equation}
The statements above provide two possible ways to pose Diophantine approximation problems,
often referred to as approximation problems of \emph{uniform} vs. \emph{asymptotic} type, namely studying  the
solvability of inequalities for all large enough values of certain parameters vs. for infinitely
many parameters, which induces respectively   limsup and liminf sets (see \cite{W12}). In this respect, it is natural to ask how much we can improve inequalities (\ref{D1}) and (\ref{D2}).

\subsubsection{\bf{Improvability of asymptotic version}}
Hurwitz's Theorem \cite{K63} shows that the approximation function $1/q$ on the right side of   (\ref{D2}) can only be improved to $1/\sqrt{5}q$.
A major obstacle for further improvement   is the existences of the  badly approximable numbers such as the Golden ratio $\frac{1+\sqrt{5}}{2}.$
Instead of considering all points,  we main concern the Diophantine properties of generic points (in a sense of measure). Precisely, for a point $\gamma\in [0,1)$ and an approximation function $\psi:\mathbb N\to \mathbb R^+$, we consider the   set
$$W(\gamma,\psi):=\Big\{x\in [0,1): |\!|qx-\gamma|\!|<\psi(q) \text{ for infinitely many } q\in \N \Big\},$$
where $|\!|\alpha|\!|:=\min\{|\alpha-m|\colon m\in \mathbb{Z}\}$ denotes the distance from $\alpha\in \R$ to the nearest integer.
We call $W(\gamma,\psi)$ a homogeneous or inhomogeneous $\psi$-well approximable set according as $\gamma=0$ or not.
Khintchine's Theorem describes the Lebesgue measure of $W(0,\psi)$ for a monotone approximation function $\psi$.
%
\begin{thm}[Khintchine, \cite{K24}]\label{K24}
Let $\psi\colon \N\to [0,\frac{1}{2})$ be a {\emph{monotonically decreasing}} function.
We have
\begin{equation*}
\mathcal{L}(W(0,\psi))=\begin{cases}
   0   & \text{if $\sum_{q=1}^{\infty}\psi(q)<\infty$}, \\
    1  & \text{if $\sum_{q=1}^{\infty}\psi(q)=\infty$},
\end{cases}
\end{equation*}
where $\mathcal{L}$ is the Lebesgue measure.
\end{thm}
We remark that, in the convergence part of Khintchine's theorem, the monotone condition on $\psi$ can be removed as the proof is an application of the first Borel-Cantelli Lemma
. However, the monotonicity is an essential  assumption in the divergence part.
Indeed, Duffin-Schaeffer \cite{DS41}  constructed a non-increasing function $\psi$ such that $\sum_q\psi(q)=\infty,$ but $W(0,\psi)$ is of Lebesgue measure 0.  Further,  consider the set
$$ {W}^{\ast}(\gamma, \psi):=\Big\{x\in [0,1)\colon  |\!|qx-\gamma|\!|^\ast<\psi(q)\text{ for infinitely many } q\in \N\Big\},$$
where
$$|\!|qx-\gamma|\!|^\ast:=\min_{\text{gcd}(p,q)=1}|qx-p-\gamma|.$$
Duffin-Schaeffer conjecture claimed that, for any $\psi\colon \N\to [0,\frac{1}{2}),$  the Lebesgue measure of the set $ {W}^{\ast}(0, \psi)$
is either 0 or 1 according as the series $\sum_{q}\frac{\phi(q)}{q}\psi(q)$ converges or diverges, where $\phi$ is the Euler's totient function.
This conjecture animated a great deal of research until it was finally proved in a breakthrough of Koukoulopoulos-Maynard \cite{KM20}.

Assuming the monotonicity of $\psi,$ Sz\"{u}sz \cite{S58} proved the inhomogeneous variant of Khintchine's Theorem. To gain a deeper understanding of the results of Khintchine  and Sz\"{u}sz, Schmidt further considered the corresponding quantitative problem.
\begin{thm}[Schmidt, \cite{S60}]
Given $x\in[0,1)$, $Q\in \N,$  and $\psi\colon \N\to [0,\frac{1}{2})$ a monotonically decreasing function,
 define
$$S(x,Q):=\sharp\big\{1\le q\le Q\colon |\!|qx-\gamma|\!|<\psi(q)\big\}.$$
Then for any $\varepsilon>0,$ for Lebesgue almost all $x\in[0,1),$
$$S(x,Q)=\Psi(Q)+O\big(\Psi^{1/2}(Q)\log^{2+\varepsilon}\Psi(Q)\big),$$
where $\Psi(Q):=\sum_{q=1}^{Q}2\psi(q)$.
\end{thm}

Without  the monotonicity assumption of $\psi,$
 the inhomogeneous version of the Duffin-Schaeffer conjecture is still a widely open question, for recent progress see \cite{AR23, Y19, Y21}. For the homogeneous case, Aistleitner-Borda-Hauke established a Schmidt-type result of Koukoulopoulos-Maynard's theorem,  and raised an open problem: to what extent the error term in Theorem \ref{ABH}  can be improved?

\begin{thm}[Aistleitner-Borda-Hauke, \cite{ABH23}] \label{ABH}
Let $\psi\colon \N\to [0,\frac{1}{2})$ be a function. Write
$$S^{\ast}(x,Q):=\sharp\big\{1\le q\le Q\colon |\!|qx|\!|^\ast<\psi(q)\big\}.$$
Let $C>0$ be an arbitrary constant. Then for almost all $x\in[0,1),$
$$S^{\ast}(x,Q)=\Psi^\ast(Q)+O\left(\Psi^\ast(Q)(\log \Psi^\ast(Q))^{-C}\right),$$
where $\Psi^\ast(Q):=\sum_{q=1}^{Q}\frac{2\phi(q)\psi(q)}{q}.$
\end{thm}

Recently, Hauke-Vazquez Saez-Walker \cite{HVW24} improved the error-term  $\big(\log\Psi^\ast(Q)\big)^{-C}$ in Theorem \ref{ABH} to $\text{exp}\big(-(\log\Psi^\ast(Q))^{\frac{1}{2}-\varepsilon}\big);$ Koukoulopoulos-Maynard-Yang
obtained an almost sharp quantitative version. 
\begin{thm} [Koukoulopoulos-Maynard-Yang, \cite{KMY24}]
Assume that $\Psi^\ast(Q)=\sum_{q=1}^{Q}\frac{2\phi(q)\psi(q)}{q} \to \infty$ as $Q\to\infty.$
Then for almost all $x\in[0,1)$ and $\varepsilon>0,$
$$S^{\ast}(x,Q)=\Psi^\ast(Q)+O\left(\Psi^\ast(Q)^{\frac{1}{2}+\varepsilon}\right).$$
\end{thm}

We now turn to a brief account of Hausdorff dimension of the $\psi$-well approximable set
$W(\gamma,\psi).$  Jarn\'{i}k Theorem \cite{J31} shows, under the monotonicity of $\psi,$  that
$$\dim_{\rm H}W(0,\psi)=\frac{2}{\tau+1}, \ \ \text{where } \tau=\liminf_{q\to\infty}\frac{-\log \psi(q)}{\log q}.$$
It is worth mentioning that Jarn\'{i}k Theorem can be deduced by combining   Theorem \ref{K24} and the mass transference principle of Beresnevich-Velani \cite{BV06}. For a general function $\psi,$ the Hausdorff dimension of $W(0,\psi)$ was studied extensively by Hinokuma-Shiga \cite{HS96}. Recently, Yu \cite{Y19} completely determined the Hausdorff dimension of $W^{\ast}(\gamma,\psi).$

\subsubsection{\bf{Improvability of uniform version}}
 The improvability of the asymptotic version leads to the study of
  the $\psi$-Dirichlet improvable set
$$D(\psi):=\left\{x\in[0,1)\colon \min_{1\le q<Q}|\!|qx|\!|\le \psi(Q) \text{ for all } Q>1\right\}.$$

As far as the metric theory of $D(\psi)$ is concerned, the continued fraction expansion plays a significant role. With the help of the Gauss transformation $T\colon [0,1) \to [0,1)$ defined by
\begin{equation*}
T(0)=0,~ T(x)= \frac{1}{x}\!\!\!\pmod1~\text{for} ~x\in(0,1),
\end{equation*}
each irrational number $x$ in $[0,1)$ can be uniquely expanded into the following form$\colon$
\begin{equation}\label{CF1}
x=\frac{1}{a_{1}(x)+\frac{1}{a_{2}(x)+\frac{1}{\ddots+\frac{1}{a_{n}(x)+T^{n}(x)}}}}
 =\frac{1}{a_{1}(x)+\frac{1}{a_{2}(x)+\frac{1}{a_{3}(x)+\ddots}}},
\end{equation}
with $a_{n}(x)=\lfloor\frac{1}{T^{n-1}(x)}\rfloor$ (here $\lfloor\cdot\rfloor$ denotes the greatest integer less than or equal to a real number and $T^0$ denotes the identity map), called the partial quotients of $x$.
For simplicity of notation, we write $(\ref{CF1})$ as
\begin{equation*}
x=[a_{1}(x),a_{2}(x),\ldots,a_{n}(x)+T^{n}(x)]=[a_{1}(x),a_{2}(x),a_{3}(x),\ldots].
\end{equation*}
We also write the truncation
$$[a_1(x), a_{2}(x), \ldots, a_{n}(x)]=:\frac{p_n(x)}{q_n(x)},$$
called the $n$-th convergent to $x.$

Davenport-Schmidt \cite{DS70} proved that a point  $x$ belongs to $D(\frac{c}{Q})$  for a certain $c\in (0, 1)$   if and only if
the partial quotients sequence $\{a_n(x)\}_{n=1}^{\infty}$ of $x$ is uniformly bounded. Thus the Lebesgue measure of $D(\frac{c}{Q})$ is zero.  Write
$$\Phi_1(q)=\frac{1}{q\psi(q)}  \ \ \text{ and } \ \ \Phi_{2}(q)=\frac{q\psi(q)}{1-q\psi(q)},$$
 and set
 \begin{equation}\label{K}
\mathcal{K}(\Phi_1):=\big\{x\in[0,1)\colon a_{n+1}(x)\ge \Phi_1(q_n(x)) \text{ for infinitely many } n\in\N\big\},
\end{equation}
 \begin{equation}\label{G}
\mathcal{G}(\Phi_2):=\big\{x\in[0,1)\colon a_{n}(x)a_{n+1}(x)\ge \Phi_2(q_n(x)) \text{ for infinitely many } n\in\N\big\}.
\end{equation}
By the optimal approximation property of the convergents, namely$\colon$
$$\min_{1\le q<q_{n+1}(x)}|\!|qx|\!|=|\!|q_n(x)\cdot x|\!|,$$
and the monotonicity of $\psi$,  we have the following inclusions about the sets $W(0,\psi)$ and $D(\psi)$$\colon$
\begin{equation}\label{inclusion}
\mathcal{K}(\Phi_1)\subset W(0,\psi)\subset \mathcal{K}(\frac{1}{3}\Phi_1),
\end{equation}
$$\mathcal{G}(\Phi_2)\subset D^{c}(\psi)\subset \mathcal{G}(\frac{1}{4}\Phi_2),$$
where $D^{c}(\psi)$, called the $\psi$-Dirichlet non-improvable set,  denotes the complement set of $D(\psi).$
It follows that $W(0,\psi)$ and $\mathcal{K}(\Phi_1),$ $D^{c}(\psi)$ and $\mathcal{G}(\Phi_2)$
share the same Khintchine-type `0-1' law and Hausdorff dimension respectively.
Based on these analysis, Kleinbock-Wadleigh \cite{KW18} completely determined the Lebesgue measure of
$D^{c}(\psi)$. Hussain \emph{et al.} \cite{HKWW18} considered the Hausdorff measure of $D^{c}(\psi)$ and showed that
$$\dim_{\rm H}\mathcal{G}(\Phi)=\dim_{\rm H}\mathcal{K}(\Phi)$$
for a non-decreasing function $\Phi\colon \N\to \mathbb{R}.$ Noting that $\mathcal{K}(\Phi)\subset \mathcal{G}(\Phi),$   it is desirable to know how large is the difference between $\mathcal{K}(\Phi)$
and $\mathcal{G}(\Phi)$ (for more  related results, one can refer to \cite{HW19,KL19,KK22,KW19,LWX23}).
\begin{thm}[Bakhtawar-Bos-Hussain, \cite{BBH20}]
Let $\Phi \colon \N\to \mathbb{R}^{+}$ be a non-decreasing function. Then
$$\dim_{\rm H}\mathcal{K}(\Phi)
  =\dim_{\rm H}\big(\mathcal{G}(\Phi) \backslash \mathcal{K}(\Phi)\big)
  =\frac{2}{\tau'+2} \ \ \text{ where } \tau'=\liminf_{q\to\infty}\frac{\log \Phi(q)}{\log q}.$$
\end{thm}

To sum up, if  $\Phi(q)=\frac{1}{q\psi(q)}$ is non-decreasing, we have
\begin{equation}\label{well-dirichlet}
\dim_{\rm H}W(0,\psi)=\dim_{\rm H}\mathcal{K}(\Phi)
=\dim_{\rm H}\mathcal{G}(\Phi)
  =\dim_{\rm H}\big(\mathcal{G}(\Phi) \backslash \mathcal{K}(\Phi)\big).
\end{equation}

\subsection{ Normal numbers, Equidistribution and Diophantine approximation on fractals}
Normal numbers were introduced by Borel in the seminal paper \cite{B09} published in 1909.
From the very beginning the concept of normality was closely related to the concept of ``randomness''.
The normality of real numbers was originally defined in terms of counting the number
of blocks of digits.  Let $b\ge 2$ be an integer. For any $x\in[0,1),$ it admits a $b$-adic expansion
$x=\sum_{n=1}^{\infty}\frac{\varepsilon_n(x)}{b^n}$ with $\varepsilon_n(x)\in\{0,1,\ldots,b-1\}$.
 For $k\in\N$ and ${\bf{x}}^{(k)}:=(x_1,\ldots,x_k)\in \{0,1,\ldots,b-1\}^{k},$ write
$$f_N^{(k)}(x,{\bf{x}}^{(k)})=\frac{\sharp\{1\le j\le N\colon (\varepsilon_j(x),\ldots,\varepsilon_{j+k-1}(x))= (x_1,\ldots,x_k) \}}{N}.$$
The number $x$ is called to be
 \emph{normal} to base $b$ if $\lim_Nf_N^{(k)}(x,{\bf{x}}^{(k)})=\frac{1}{b^{k}}$ for any  $k\ge1$
and ${\bf{x}}^{(k)}\in \{0,1,\ldots,b-1\}^{k}$, and
 (\emph{absolutely}) \emph{normal} if it is normal to every base $b\ge 2.$
 Borel \cite{B09}  first showed that Lebesgue almost every $x\in[0,1)$ is normal. And thus  any subset of $[0,1)$ of positive Lebesgue measure must contain normal numbers; it is natural to ask  what conditions can guarantee that a fractal set $E$ (usually of Lebesgue measure zero) contains normal numbers.
We note that the condition $\dim_{\rm H}E>0$ is not sufficient by considering the  Cantor middle-thirds set.

A real-valued sequence $(x_n)_{n\in\N}$ is called \emph{equidistribution} or \emph{uniformly distributed modulo one} if for each sub-interval $[a,b]\subseteq[0,1]$ we have
\begin{equation}\label{equi}
\lim_{N\to\infty}\frac{\sharp\big\{1\le n\le N\colon x_n \text{ mod } 1\in [a,b]\big\}}{N}=b-a.
\end{equation}
The notation of equidistribution has been studied intensively since the beginning of the 20th century, originating in Weyl's seminal paper \cite{Weyl16}. This topic has developed into an important area of mathematics, with many deep connections to fractal geometry, number theory, and probability theory.
Generally, it is a challenging problem to determine whether a given sequence of is equdistribution.  For example, we do not know whether the sequence $\{(3/2)^n\}$ is equdistribution or not. This problem seems to be completely out of reach for current methods \cite{D09,FLP95},
even though Weyl (\cite{B12,KN74}) established his famous criterion, which reduces the  equidistribution problem  to  bounds of some exponential sums.

 The following theorem of Davenport \emph{et al.} \cite{DEL63} established the generic distribution properties of a sequence $(q_nx)_{n\in \N}$ with $x$ is restricted to a subset $E$ of $[0,1].$  As is customary, we write $e(x)=\text{exp}(2\pi ix).$

\begin{thm}[Davenport-Erd\"{o}s-Leveque, \cite{DEL63}]\label{DEL}
Let $\mu$ be a Borel probability measure supported on a subset $E\subseteq[0,1].$ Let $\mathcal{A}=(q_n)_{n\in\N}$
be a sequence of reals satisfying
$$\sum_{N=1}^{\infty}\frac{1}{N}\int\Big|\frac{1}{N}\sum_{n=1}^{N}e(kq_nx)\Big|^{2}\d\mu(x)<\infty$$
for any $k\in \mathbb{Z}\setminus\{0\},$ then for $\mu$ almost all $x\in E$ the sequence $(q_nx)_{n\in \N}$ is equdistributed.
\end{thm}
Recalling that the Fourier transform of a non-atomic   measure $\mu$ is defined by
$$\widehat{\mu}(t):=\int e(-tx)\d\mu(x) \quad (t\in \mathbb{R}),$$
we may rewrite  the series in Theorem \ref{DEL} as
\begin{equation}\label{criterion}
\sum_{N=1}^{\infty}\frac{1}{N}\int \big|\frac{1}{N}\sum_{n=1}^{N}e(kq_nx)\big|^2 \d\mu(x)
=\sum_{N=1}^{\infty}\frac{1}{N^{3}}\sum_{m,n=1}^{N}\widehat{\mu}(k(q_n-q_m)).
\end{equation}
Wall \cite{Wall50} proved that a real number $x$ is normal to integer base $b\ge2$ if
and only if the sequence $(b^nx)_{n\in\N}$ is equdistributed. Along this line, Theorem \ref{DEL} and (\ref{criterion}) show that  Lebesgue almost all reals are normal to every integer base, as already proved by Borel.
%
%
Also, we may transfer the existence of normal numbers in a fractal set to an equidistribution problem of lacunary sequences.
Recall that a sequence $\mathcal{A}=(q_n)_{n\in\N}$ is said to be \emph{lacunary} if there exists  $C>1$ such that $\mathcal{A}$ satisfies  the classical Hadamard gap condition
\begin{equation}\label{gap}
\frac{q_{n+1}}{q_n}\ge C \ \ (n\in \N).
\end{equation}
Pollington \emph{et al.} \cite{PVZZ22}  showed that if
\begin{equation} \label{normal}
\widehat{\mu}(t)=O\left((\log\log|t|)^{-(1+\varepsilon)}\right)
\end{equation}
with $\varepsilon>0$, then for $\mu$ almost every $x\in E,$  $(\ref{equi})$  holds with $x_n=q_nx$ for a  lacunary sequence $\mathcal{A}=\{q_n\}$ of natural numbers. It follows that $\mu$ almost every  $x\in E$ is a normal number.
Hence the existence of normal numbers in a fractal set can be deduced from the  Fourier decay of a measure supported on the set.
We   refer to \cite{ABS22,ARS22, C59, DGW02,HS15,P22,S60} for some recent works on the logarithmic and polynomial Fourier decay for fractal measures.

We now turn to another aspect of the equidistribution theory which is relevant to the hitting problems.
Let   $\mathbb{B}:=B(\gamma,r)$ denote the ball with centere $\gamma\in [0,1]$ and radius $r\le 1/2.$
Theorem \ref{DEL} implies that,  for $\mu$-almost all $x\in E,$ the sequence $(q_nx)_{n\in \N}$  hits the ball $\mathbb{B}$ for the expected number of times, namely:
$$\lim_{N\to\infty}\frac{\sharp\{1\le n\le N\colon |\!|q_nx-\gamma|\!| \le r\}}{N}=2r.$$
Instead of a fixed $r,$ we consider the situation in which $r$ is allowed to shrink with time. In view of this, let $\psi\colon \R \to (0,1)$
be a real function, and set the counting function
$$R(x,N):=R(x,N; \gamma, \psi, \mathcal{A}):=\sharp\{1\le n\le N\colon |\!|q_nx-\gamma|\!|\le \psi(q_n)\}.$$
As alluded to in the definition, we will often simply write $R(x,N)$ for $R(x,N; \gamma, \psi, \mathcal{A})$ since the other dependencies are usually fixed and  will be clear from
the context.

Pollington \emph{et al.} \cite{PVZZ22} studied the quantitative property of the counting function $R(x,N).$

\begin{thm}[Pollington-Velani-Zafeiropoulos-Zorin, \cite{PVZZ22}]\label{PVZZ}
Let $\mathcal{A}=(q_n)_{n\in\N}$ be a lacunary sequence of natural numbers.
Let $\gamma\in [0,1]$ and $\psi\colon \N\to (0,1)$.
 Let $\mu$ be a probability measure supported on a subset $E$ of $[0,1]$. Suppose that
 there exists   $A>2$ such that
$$\widehat{\mu}(t)=O\left((\log |t|)^{-A}\right) ~~\text{as }  |t|\to \infty.$$
Then for any $\varepsilon>0,$ we have that
$$R(x,N)=2\Psi(N)+O\Big(\Psi(N)^{2/3}(\log(\Psi(N)+2))^{2+\varepsilon}\Big)$$
for $\mu$-almost all $x\in E,$ where $\Psi(N)=\sum_{n=1}^{N}\psi(q_n).$
\end{thm}

In  Theorem \ref{PVZZ}, the authors mentioned that the associated error terms  would be able to improve. We continued the study for a  sequence $\mathcal{A}$ of real numbers (instead of natural numbers).
Note that the factorization and coprimeness properties of the natural numbers played a significant role in the proof of Theorem \ref{PVZZ}. 

\begin{thm}\label{TZ}
Let $\mathcal{A}=(q_n)_{n\in\N}$ be a real-valued lacunary sequence. Let $\gamma\in [0,1]$
and $\psi\colon \R\to (0,1)$.
Let $\mu$ be a probability measure supported on a subset $E$ of $[0,1]$. Suppose that   there exists $A>2$ such that
$$\widehat{\mu}(t)=O\Big((\log |t|)^{-A}\Big) ~~\text{as }  |t|\to \infty.$$
Then for any $\varepsilon>0,$ we have that
$$R(x,N)=2\Psi(N)+O\Big(\Psi(N)^{1/2}(\log(\Psi(N)+2))^{\frac{3}{2}+\varepsilon}\Big)$$
for $\mu$-almost all $x\in E,$ where $\Psi(N)=\sum_{n=1}^{N}\psi(q_n).$
\end{thm}

\begin{rem}
The exponent $\frac{1}{2}$ in the error term is optimal in Theorem \ref{TZ}.
\end{rem}

Motivated by the metric theory of Diophantine approximation on manifolds, we consider the Diophantine properties of points which are restricted to a sub-manifold of $[0,1].$ Given a real number $\gamma\in[0,1],$
a real function $\psi\colon \mathbb{R}\to (0,1)$ and a sequence $\mathcal{A}=(q_n)_{n\in\N},$ set
$$W_\mathcal{A}(\gamma, \psi):=\{x\in [0,1]\colon |\!|q_nx-\gamma|\!| \le \psi(q_n) \text{  for infinitely many } n \in \N\}.$$
Our goal is to obtain an analogue of Khintchine-type Theorem for the size of $W_\mathcal{A}(\gamma, \psi)\cap E,$ where $E$ is a subset of $[0,1].$  A direct corollary of Theorem \ref{TZ} implies that $\mu$-almost all $x\in E,$ the quantity $R(x,N)$ is bounded if $\Psi(N)$ is bounded, and  tends to infinity if $\Psi(N)$ tends to infinity.

\begin{cor}\label{TZ1}
 Let $\mathcal{A}=(q_n)_{n\in\N}$ be a real-valued lacunary sequence. Let $\gamma\in [0,1]$
and $\psi\colon \R\to (0,1)$. Let $\mu$ be a probability measure supported on a subset $E$ of $[0,1].$ Suppose there exists   $A>2$ such that
$$\widehat{\mu}(t)=O\Big((\log |t|)^{-A}\Big) ~~\text{as }  |t|\to \infty.$$
Then
\begin{equation*}
\mu(W_\mathcal{A}(\gamma, \psi)\cap E)=\begin{cases}
   0   & \text{if  $\sum_{n=1}^{\infty}\psi(q_n)<\infty$}, \\
    1  & \text{if  $\sum_{n=1}^{\infty}\psi(q_n)=\infty$}.
\end{cases}
\end{equation*}
\end{cor}

\begin{rem}
No monotonicity assumption on $\psi$ is necessary in Corollary \ref{TZ1}.

\end{rem}

\subsection{Fourier dimension and Salem set}
The Fourier dimension of a Borel set $E\subseteq \mathbb{R}$ is defined by
$$\dim_{\rm F}E=\sup\big\{s\in[0,1]\colon \exists \mu\in \mathcal{P}(E) \text{ such that } |\widehat{\mu}(\Theta)|\le C_s(1+|\Theta|)^{-s/2}\big\}.$$
Here and elsewhere $\mathcal{P}(E)$ denotes the set of Borel probability measure on $\mathbb{R}$ whose support is contained in $E.$  Hence, by the definition of Fourier dimension and condition (\ref{normal}), we obtain that:
$$\text{If } \dim_{\rm F}E>0, \text{ then } E \text{ must contains normal numbers.}$$
A classical construction of a measure with polynomial Fourier decay by Kaufman \cite{K80}, later updated by Queff\'{e}le-Ramar\'{e} \cite{QR03}, shows that the set of badly approximable numbers has positive Fourier diemension. Chow \emph{et al.} \cite{CZZ23} further developed the ideas from \cite{K80} and established an inhomogeneous variant of Kaufman's measure. Recently,
Fraser-Wheeler \cite{FW, FW23} proved that the set of exact approximation order has positive Fourier diemension. Some Fourier dimension estimates of sets of well-approximable matrices are given in \cite{CH24,H19, HY23}.

Fourier dimension is closely related to Hausdorff dimension. Indeed, Frostman's lemma \cite{M95} states that the Hausdorff dimension of a Borel set $E\subseteq \mathbb{R}$ is
$$\dim_{\rm H}E=\sup\left\{s\in[0,1]\colon \exists \mu\in \mathcal{P}(E) \text { such that }\int |\widehat{\mu}(\Theta)|^{2}|\Theta|^{s-1}\d\mu<\infty\right\}.$$
It follows that
$$\dim_{\rm F}E\le \dim_{\rm H}E$$
for every Borel set $E\subseteq \mathbb{R}.$ See \cite{M15} for more information.

Generally the Hausdorff and Fourier dimensions of a set are distinct. For example, every $(n-1)$-dimensional hyperplane in $\mathbb{R}$ has Hausdorff dimension $n-1$ and Fourier dimension 0. The Cantor middle-thirds set has Hausdorff dimension $\log2/\log3$ and Fourier dimension 0.  
A set $E\subseteq\mathbb{R}$ is called  a \emph{Salem set} if $\dim_{\rm F}E=\dim_{\rm H}E.$  Salem \cite{S51} proved that for every $s\in[0,1]$ there exists a Salem set with dimension $s$ by constructing random Cantor-type sets. Kahane \cite{K66} showed that for every $s\in[0,n]$ there exists a Salem set in $\mathbb{R}^{n}$ with dimension $s$ by considering images of Brownian motion. For other random Salem sets the readers are referred to \cite{B98, CS17, E16, LP09, SS18} and references therein.

We now focus on finding explicit Salem set.  To the best of our knowledge,  explicit Salem sets are much more rare.  We list some  Salem sets of explicit version:

\begin{itemize}
\item Kaufman \cite{K81} (1981), Bluhm \cite{B98} (1996):

If $\tau\ge1$,  the well-approximable set $W(0, q\mapsto q^{-\tau})$ is a Salem  of dimension  $\frac{2}{\tau+1}$;
\medskip

\item Hambrook \cite{H17} (2017):

Identify $\mathbb{R}^{2}$ with  $\mathbb{C}$, and the lattice $\mathbb{Z}^{2}$ as the ring of the Gaussian integers $\mathbb{Z}+i\mathbb{Z}.$ And thus, for $q\in \mathbb{Z}^{2}$ and
$x\in \mathbb{R}^{2},$ $qx$ can be regarded as a product of complex numbers.
If $\tau\ge1$, the set
\begin{center}
$\big\{x\in \mathbb{R}^{2}\colon |qx-p|_{\infty}\le (|q|_{\infty})^{-\tau} \text { for i.m. }(q, p)\in \mathbb{Z}^{2}\times \mathbb{Z}^{2}\big\}$
\end{center}
is a Salem of dimension  $\min\{\frac{4}{\tau+1}, 2\}$, where $|x|_{\infty}$ denotes the max-norm of  $x\in\mathbb{R}^{2}$ and `i.m.' means `infinitely many';
\medskip

\item Fraser-Hambrook \cite{FH23} (2023):

Let $K$  be a number field of degree $n$ with  $\mathbb{Z}_K$  its ring of integers. Fix  an integral basis $\{\omega_1,\ldots,\omega_n\}$   for $\mathbb Z_K.$
The mapping  $(q_1,\ldots,q_n)\mapsto \sum_{i=1}^{n}q_i\omega_i$ establishes an identification between $\mathbb{Q}^{n}$ and $K$, as well as
$\mathbb{Z}^{n}$ and $\mathbb{Z}_{K}.$  When $\tau\ge1$, the set
\begin{center}
$\Big\{x\in \mathbb{R}^{n}\colon \Big|x-\frac{p}{q}\Big|_{2}\le (|q|_{2})^{-\tau-1} \text { for i.m. }(p, q)\in \mathbb{Z}^{n}\times \mathbb{Z}^{n}\Big\}$
 \end{center}
is a Salem of dimension  $\frac{2n}{\tau+1}$, where $|x|_{2}$ denotes the 2-norm  of
$x\in \mathbb{R}^{n}$ for $n\in \N$.
\end{itemize}

\begin{cor}\label{TZ4}
For  $\tau\ge1,$ let $\Phi(q)=q^{\tau-1}$.  Then $\mathcal{K}(3\Phi)$ and $\mathcal{G}(3\Phi)$
are Salem sets of dimension $\frac{2}{\tau+1}$, where $\mathcal{K}(3\Phi)$ and $\mathcal{G}(3\Phi)$  are defined as (\ref{K}) and (\ref{G}) respectively.
\end{cor}
\begin{proof}
By (\ref{inclusion}), (\ref{well-dirichlet}) and the fact $\mathcal{K}(3\Phi)\subset \mathcal{G}(3\Phi),$ we have
$$\dim_{\rm F}W(0, q\mapsto q^{-\tau})\le \dim_{\rm F}\mathcal{K}(3\Phi)\le\dim_{\rm F}\mathcal{G}(3\Phi)\le \dim_{\rm H}W(0, q\mapsto q^{-\tau}).$$
We complete the proof by recalling $W(0, q\mapsto q^{-\tau})$ is a Salem set.
\end{proof}

\begin{rem}
By Thoerem 1.2 in \cite{LP09}, we remark that if $\tau$ is sufficiently close to 1, then both $\mathcal{K}(3\Phi)$ and $\mathcal{G}(3\Phi)$ contain non-trivial 3-term arithmetic progressions.
\end{rem}

Corollary \ref{TZ4} motivates  the following questions:
\begin{itemize}
\item The sets $\mathcal{K}(3\Phi)$ and $\mathcal{G}(3\Phi)$ must contain normal numbers.  Does the difference   $\mathcal{G}(3\Phi) \backslash \mathcal{K}(3\Phi)$ contain normal numbers?
\smallskip

\item When $\Phi(q)=q^{\tau-1},$  is $\mathcal{G}(3\Phi) \backslash \mathcal{K}(3\Phi)$  a Salem set?
\smallskip

\item For a general function $\Phi,$ are $\mathcal{K}(3\Phi)$,  $\mathcal{G}(3\Phi)$ or  $\mathcal{G}(3\Phi) \backslash \mathcal{K}(3\Phi)$  Salem?
\end{itemize}

We obtain a partial   result.
\begin{thm}\label{TZ2}
Let $\Phi\colon [1,\infty) \to \mathbb{R}^{+}$ be a non-decreasing function satisfying
$$\liminf_{q\to\infty}\frac{\log\Phi(q)}{\log q}=:\tau <\frac{\sqrt{73}-3}{8}.$$
Then $\mathcal{G}(3\Phi) \backslash \mathcal{K}(3\Phi)$ has positive Fourier dimension.
\end{thm}

\begin{rem}
In the proof of Theorem $\ref{TZ2},$
the similarly argument as in subsection \ref{general} can apply to extend the main result Theorem 1.2 in \cite{FW23} by reducing the constraints on the approximation function $\Phi.$  To be accurate, assuming that the function $\Phi$ satisfies the conditions in Theorem \ref{TZ2},  we have that the
set of exact approximation order
$$\mathrm{Exact}(\Phi)=\Big\{x\in[0,1)\colon \limsup_{n\to\infty}\frac{\log a_{n+1}(x)}{\log \Phi(q_n(x))}=1\Big\} $$ is of positive Fourier dimension.
\end{rem}





\section{Preliminaries}
\subsection{Continued fractions}
This subsection is devoted to recalling some elementary properties of continued fractions. For more information, the readers are referred to \cite{HW79,K63,S80}.

For an irrational number $x\in[0,1)$ with continued fraction expansion (\ref{CF1}), the sequences
$\{p_n(x)\}_{n\ge 0},$ $\{q_n(x)\}_{n\ge0}$ satisfy the following recursive relations \cite{K63}:
\begin{equation*}
  p_{n+1}(x)=a_{n+1}(x)p_{n}(x)+p_{n-1}(x),~~q_{n+1}(x)=a_{n+1}(x)q_{n}(x)+q_{n-1}(x),
\end{equation*}
with the conventions that $(p_{-1}(x),q_{-1}(x))=(1,0),$ $(p_0(x),q_0(x))=(0,1).$
 Clearly, $q_{n}(x)$ is determined by $a_{1}(x),\ldots,a_{n}(x),$ so we also write
            $q_{n}(a_{1}(x),\ldots,a_{n}(x))$
instead of $q_n(x)$. We write $a_{n}$ and $q_n$ in place of $a_{n}(x)$ and $q_n(x)$ when no confusion can arise.

\begin{lemma} [Khintchine, \cite{K63}]\label{K1}
For  $(a_{1},\ldots,a_{n})\in \mathbb{N}^{n}$, we have$\colon$

{\rm{(1)}} $q_{n}\geq 2^{\frac{n-1}{2}},$ and $\prod\limits_{k=1}^{n}a_{k}\leq q_{n}\leq \prod\limits_{k=1}^{n}(a_{k}+1).$

{\rm{(2)}} For    $k\ge 1,$
\begin{equation*}
1\leq \frac{q_{n+k}(a_{1},\ldots,a_{n},a_{n+1},\ldots,a_{n+k})}
{q_{n}(a_{1},\ldots,a_{n})q_{k}(a_{n+1},\ldots,a_{n+k})}\leq 2.
\end{equation*}
\end{lemma}

A basic cylinder of order $n$ is a set of the form
       $$I_{n}(a_{1},\ldots,a_{n}):=\{x\in[0,1)\colon a_{k}(x)=a_{k}, 1\leq k\leq n\};$$
the basic cylinder of order $n$ containing $x$ will be denoted by  $I_{n}(x)$, i.e.,
       $I_{n}(x)=I_{n}(a_{1}(x),\ldots,a_{n}(x))$.

\begin{lemma} [Khintchine, \cite{K63}]\label{K2}
For   $(a_{1},\ldots,a_{n})\in \mathbb{N}^{n},$ we have
\begin{equation*}
\frac{1}{2q_{n}^{2}}\leq|I_{n}(a_{1},\ldots,a_{n})|=\frac{1}{q_{n}(q_{n}+q_{n+1})}\leq\frac{1}{q_{n}^{2}}.
\end{equation*}
\end{lemma}

The next lemma describes the distribution of basic cylinder $I_{n+1}$ of order $n+1$ inside an $n$-th  basic interval $I_{n}.$

\begin{lemma}[Khintchine, \cite{K63}]\label{cylinder-dis}
 Let $I_{n}(a_{1},\ldots,a_{n})$ be a basic cylinder of order $n,$ which is partitioned into sub-intervals $I_{n+1}(a_{1},\ldots,a_{n},a_{n+1})$ with $a_{n+1}\in \mathbb{N}.$
When $n$ is odd (resp. even), these sub-intervals are positioned from left to right (resp. from right to left), as $a_{n+1}$ increases.

\end{lemma}

We conclude this subsection by citing a dimensional result on continued fractions.
Let ${\bf{Bad}}(N)$ be the set consisting of all points in [0,1) whose partial quotients are not great than $N,$ i.e.,
\begin{equation*}
{\bf{Bad}}(N)=\{x\in[0,1)\colon 1\leq a_{n}(x)\leq N \text{ for } n\geq 1\}.
\end{equation*}

\begin{lemma}[Jarn\'{i}k, \cite{J28}]\label{J}
For $N\geq 8,$
$$1-\frac{1}{N\log 2}\le \dim_{\rm H}{\bf{Bad}}(N)\le 1-\frac{1}{8N\log N}.$$
\end{lemma}

\subsection{Oscillatory integrals}
 In this subsection, we cite three lemmas on the oscillatory integrals for later use.
 The first van der Corput-type inequality is useful for nonstationary phases.

 \begin{lemma}[Kaufman, \cite{K80}]\label{Kaufman1}
 If  $f$ is a  $C^2$ function  on [0,1], and satisfies that $|f'(x)|\ge A$ and $|f''(x)|\le B,$ then
 $$\left|\int_{0}^{1}e(f(x))\d x\right|\le\frac{1}{A}+\frac{B}{A^{2}}.$$
 \end{lemma}

 The second van der Corput-type inequality applies when $f'(x)$ vanishes at some points in [0,1] in question.
 \begin{lemma}[Kaufman, \cite{K80}]\label{Kaufman2}
 If $f$ is  $C^2$   on [0,1], and $f'(t)=(C_1x+C_2)g(x) $ where $g$ satisfies $|g(x)|\ge A$
 and $|g'(x)|\le B$ with $B\ge A,$ then we have
 $$\left|\int_{0}^{1}e(f(x))\d x\right|\le \frac{6B}{A^{3/2}C_1^{1/2}}.$$
 \end{lemma}

 The last is a comparison lemma which enables us to compare an integral with respect to a general measure to an integral with respect to Lebesgue measure. The original ideas for the proofs date back to Kaufman \cite{K80} (see also \cite{FW23,QR03}).
 \begin{lemma}\label{Kaufman3}
 Let $F$ be a $C^2$ function on [0,1] satisfying $|F(x)|\le 1$ and $|F'(x)|\le M$, and write  $m_2=\int_{0}^{1}|F(x)|^{2}\d x.$ Let $\mu$ be a Borel probability measure on [0,1], and let $\Lambda(h)$
 be the maximum $\mu$-measure of all intervals $[t,t+h]$ of length $h$. Then we have for all $r>0,$
 $$\int_{0}^{1}|F(x)|\d\mu(x)\le 2r+\Lambda\Big(\frac{r}{M}\Big)\cdot\Big(1+\frac{m_2M}{r^{3}}\Big).$$
 \end{lemma}

\section{
Establishing Theorem \ref{TZ}}
Before presenting the proof of Theorem \ref{TZ}, we cite a quantitative version of the Borel-Cantelli lemma.

\begin{lemma}[Philipp, \cite{P67}]\label{Philipp}
Let $(X,\mathcal{B},\mu)$ be a probability space, let $(f_n(x))_{n\in \N}$ be a sequence of non-negative $\mu$-measurable functions defined
on $X,$ and let $(f_n)_{n\in \N},$ $(\phi_n)_{n\in \N}$ be sequences of reals such that
$$0\le f_n\le \phi_n \ \ \ (n=1,2,\ldots).$$
Suppose that for arbitrary $a, b\in \N$ with $a<b,$ we have
$$\int_{X}\Big(\sum_{n=a}^{b}(f_n(x)-f_n)\Big)^{2}\d\mu(x)\le C\sum_{n=a}^{b}\phi_n$$
for an absolute constant $C>0.$ Then, for any given $\varepsilon>0,$
$$\sum_{n=1}^{N}f_n(x)=\sum_{n=1}^{N}f_n+O\Big(\Psi(N)^{1/2}(\log\Psi(N))^{\frac{3}{2}+\varepsilon}+\max_{1\le n\le N}f_n\Big)$$
for $\mu$-almost all $x\in X,$ where $\Psi(N)=\sum_{n=1}^{N}\phi_n.$
\end{lemma}

The rest of this section is devoted to proving Theorem \ref{TZ}.  We first fix some notation.
Let $X=[0,1],$ $f_n(x)=\mathbb{I}_{E_{q_n}^{\gamma}}(x)$ and $f_n=2\psi(q_n),$ where $\mathbb{I}_{E_{q_n}^{\gamma}}$ is the indicator  function of the set
$$E_{q_n}^{\gamma}:=\big\{x\in[0,1]\colon |\!|q_nx-\gamma|\!|\le \psi(q_n)\big\}.$$
Hence
$$R(x,N)=\sum_{n=1}^{N}f_{n}(x).$$
Furthermore, it is readily verified  for $a, b\in \N$ with $a<b$ that
$$\Big(\sum_{n=a}^{b}(f_n(x)-f_n)\Big)^{2}=\sum_{n=a}^{b}f_n(x)+2\sum_{a\le m<n\le b}f_m(x)f_n(x)+
 \Big(\sum_{n=a}^{b}f_n\Big)^{2}-2\sum_{n=a}^{b}f_n\cdot\sum_{n=a}^{b}f_n(x),$$
and thus
\begin{equation}\label{expect}
\begin{split}
\int_{0}^{1}\Big(\sum_{n=a}^{b}(f_n(x)-f_n)\Big)^{2}\d\mu(x)=&\sum_{n=a}^{b}\mu(E_{q_n}^{\gamma})
+2\sum_{a\le m<n\le b}\mu(E_{q_n}^{\gamma}\cap E_{q_m}^{\gamma})
\\&-4\sum_{n=a}^{b}\psi(q_n)\Big(\sum_{n=a}^{b}\mu(E_{q_n}^{\gamma})-\sum_{n=a}^{b}\psi(q_n)\Big).
\end{split}
\end{equation}

\subsection{Estimation of $\sum_{n=a}^{b}\mu(E_{q_n}^{\gamma})$}
We set
\begin{equation*}
\mathbb{I}_{\psi(q_n), \gamma}(x)=
\begin{cases}
   1  & \text{if $ |x-\gamma|\le \psi(q_n)$}, \\
   0 & \text{otherwise},
\end{cases}
\end{equation*}
i.e., $\mathbb{I}_{\psi(q_n), \gamma}(x)$ is the indicator function of the interval
$[l_n, r_n]$ with $l_n=\gamma-\psi(q_n)$ and $r_n=\gamma+\psi(q_n)$. Hence $f_n(x)=\sum_{k\in \mathbb{Z}}\mathbb{I}_{\psi(q_n), \gamma}(q_nx+k).$
We require an approximation function $g(x)$ to $\mathbb{I}_{\psi(q_n), \gamma}(x)$  such that $\widehat{g}(x)$ vanishes outside some interval,
say $[-n^{3},n^{3}]$, and thus we also have
$$\sum_{k=-n^{3}}^{n^{3}}\widehat{g}(k)e(kx)$$
as an approximation to $\mathbb{I}_{\psi(q_n), \gamma}(x).$

We begin the construction of the approximation by writing
\begin{equation*}
 \mathbb{I}^{\ast}(x)=\begin{cases}
    ~1  & \text{if $x\ge0$}, \\
   -1   & \text{if $x<0$}.
\end{cases}
\end{equation*}
Hence   $\mathbb{I}_{\psi(q_n), \gamma}(x)=\frac{1}{2}\big[\mathbb{I}^{\ast}\big(H(x-l_n)\big)+
\mathbb{I}^{\ast}\big(H(r_n-x)\big)\big]$
for any $H>0.$  Beurling and Selberg  \cite{V85} showed that the function
$$F(z)=\Big[\frac{\sin \pi z}{\pi}\Big]^{2} \Big[\sum_{n=0}^{\infty}(z-n)^{-2}-\sum_{n=1}^{\infty}(z+n)^{-2}+\frac{2}{z}\Big]$$
satisfies that
$$F(x)\ge \mathbb{I}^{\ast}(x), \ \  \int_{-\infty}^{+\infty}[F(x)-\mathbb{I}^{\ast}(x)]\d x=1,$$
and for any $\alpha, \beta\in\mathbb R,$
$$\int_{-\infty}^{+\infty}[F(x-\alpha)+F(x-\beta)]e(-tx)\d x=0 \ \ \text{when } |t|\ge 1.$$

We define the desired functions as
\begin{align*}
  g_1(x)&=~\frac{1}{2}\Big[F\big(n^{3}(x-l_n)\big)+ F\big(n^{3}(r_n-x)\big)\Big],\\
  g_2(x)&=-\frac{1}{2}\Big[F\big(n^{3}(l_n-x)\big)+ F\big(n^{3}(x-r_n)\big)\Big].
\end{align*}
We list some properties of  $g_i(x) (i=1, 2)$;  the reader may find more details in Lemmas 2.3-2.5 of the book \cite{B86}.
\begin{enumerate}
  \item $g_i(x)\in L^{1}(\mathbb{R});$

  \smallskip

  \item $g_2(x)\le \mathbb{I}_{\psi(q_n), \gamma}(x)\le g_1(x);$

  \smallskip

  \item $\widehat{g_i}(m)=0$ \text{ for } $|m|> n^{3};$

  \smallskip

  \item $\int_{-\infty}^{+\infty} |g_i(x)-\mathbb{I}_{\psi(q_n), \gamma}(x)|\d x=\frac{1}{n^{3}}.$
\end{enumerate}
These properties together with Poisson summation Formula yield that
$$\sum_{k=-n^3}^{n^3}\widehat{g}_2(k)e(kq_nx)\le f_n(x)=\sum_{k\in \mathbb{Z}}\mathbb{I}_{\psi(q_n), \gamma}(q_nx+k)\le \sum_{k=-n^3}^{n^3}\widehat{g}_1(k)e(kq_nx),$$
$$\widehat{g}_1(0)=2\psi(q_n)+\frac{1}{n^{3}},
~~\widehat{g}_2(0)=2\psi(q_n)-\frac{1}{n^{3}}.$$
Moreover,  for $k\in \mathbb{Z}\setminus\{0\},$ we estimate
\begin{align*}
|\widehat{g}_i(k)|
& =\left|\int_{-\infty}^{+\infty}\mathbb{I}_{\psi(q_n), \gamma}(x)e(-kx)\d x+
   \int_{-\infty}^{+\infty}[g_i(x)-\mathbb{I}_{\psi(q_n), \gamma}(x)]e(-kx)\d x\right|\\
&\le \min\left\{\frac{1}{|k|}, 2\psi(q_n)\right\}+\frac{1}{n^{3}},
\end{align*}
and thus
\begin{equation*}
\begin{split}
\sum_{n=a}^{b}\mu(E_{q_n}^{\gamma})
&=\sum_{n=a}^{b}\int_{0}^{1}f_n(x)\d\mu(x)
 \le\sum_{n=a}^{b}\sum_{k=-n^3}^{n^3}\int_{0}^{1}\widehat{g}_1(k)e(kq_nx)\d\mu(x)\\
&=\sum_{n=a}^{b}\left[2\psi(q_n)+\frac{1}{n^{3}}\right]+ \sum_{n=a}^{b}\sum_{-n^{3}\le k\le n^{3}, k\neq 0}\widehat{g}_1(k)\widehat{\mu}(-kq_n)\\
&=2\sum_{n=a}^{b}\psi(q_n)+O\left(\sum_{n=a}^{b}(\log q_n)^{-A}\sum_{1\le k\le n^{3}}\Big(\frac{1}{k}+\frac{1}{n^{3}}\Big)\right)\\
&\stackrel{(\ref{gap})}{=}2\sum_{n=a}^{b}\psi(q_n)+O(1).
\end{split}
\end{equation*}
  Argued in a similar way (with $g_2$ instead of $g_1$), we obtain the reversed inequality. And thus we reach that
$$\sum_{n=a}^{b}\mu(E_{q_n}^{\gamma})=2\sum_{n=a}^{b}\psi(q_n)+O(1).$$

\subsection{Estimation of $\sum_{n=a}^{b}\mu(E_{q_n}^{\gamma}\cap E_{q_m}^{\gamma})$}
For   $m, n\in \N$ with $m<n,$ we have
$$\mu(E_{q_n}^{\gamma}\cap E_{q_m}^{\gamma})
= \int_{0}^{1}f_n(x)f_m(x)\d\mu(x)
\le\sum_{\substack{-n^{3}\le s,t\le n^{3}}} \widehat{g}_1(s)\widehat{g}_1(t)\widehat{\mu}(-(sq_n+tq_m)).$$
We will divide the above summation into four parts.

\noindent\underline{\textbf{Case 1$\colon$} $s=t=0.$}   We have
$$\widehat{g}_1(0)\widehat{g}_1(0)\widehat{\mu}(0)
=4\psi(q_n)\psi(q_m)+\frac{2\psi(q_n)}{n^{3}}+\frac{2\psi(q_m)}{n^{3}}+\Big(\frac{1}{n^{3}}\Big)^{2}.$$

\noindent\underline{\textbf{Case 2$\colon$} $s\neq 0,  t=0.$}  We deduce
\begin{align*}
\sum_{\substack{-n^{3}\le s\le n^{3} \\ s\neq 0}}\widehat{g}_1(s)\widehat{g}_1(0)\widehat{\mu}(-sq_n)&\le
\sum_{\substack{-n^{3}\le s\le n^{3} \\ s\neq 0}}\left[\min\Big(\frac{1}{|s|}, 2\psi(q_n)\Big)+\frac{1}{n^{3}}\right](2\psi(q_m)+\frac{1}{n^{3}})\widehat{\mu}(-sq_n)\\
&\ll \Big(2\psi(q_m)+\frac{1}{n^{3}}\Big)\frac{\log n}{n^{A}}.
\end{align*}

\noindent\underline{\textbf{Case 3$\colon$} $s= 0,  t \neq 0.$}
Similar argument to \textbf{Case 2} shows that
$$\sum_{\substack{-n^{3}\le t\le n^{3} \\ t\neq 0}}\widehat{g}_1(0)\widehat{g}_1(t)\widehat{\mu}(-tq_m)
\ll \left(2\psi(q_n)+\frac{1}{n^{3}}\right)\frac{\log n}{n^{A}}.$$

\noindent\underline{\textbf{Case 4$\colon$} $s\neq 0,  t \neq 0.$} We calculate
\begin{equation*}
\begin{split}
&\sum_{\substack{-n^{3}\le s\le n^{3} \\ s\neq 0}}\sum_{\substack{-n^{3}\le t\le n^{3} \\ t\neq 0}}\widehat{g}_1(s)\widehat{g}_1(t)\widehat{\mu}(-(sq_n+tq_m))\\
&\le \sum_{\substack{-n^{3}\le s,t\le n^{3}  \\ st\neq 0}}\left[\min\Big(\frac{1}{|s|}, 2\psi(q_n)\Big)+\frac{1}{n^{3}}\right]
\left[\min\Big(\frac{1}{|t|}, 2\psi(q_n)\Big)+\frac{1}{n^{3}}\right]\widehat{\mu}(-(sq_n+tq_m)).
\end{split}
\end{equation*}

The latter summation is divided into two parts as follows.

(1) When $s$ and $t$ have the same sign,  we have the summation over such kind of $(s,t)$
$$<\sum_{\substack{-n^{3}\le s, t\le n^{3}   \\ st> 0}}\left(\frac{1}{|s|}+\frac{1}{n^{3}}\right)\left(\frac{1}{|t|}+\frac{1}{n^{3}}\right)n^{-A}
\ll n^{-A}(\log n)^{2}.$$

(2) When $s$ and $t$ have the opposite signs, we estimate the summation 
 by considering two subcases.
\smallskip

(2.1) $|sq_n-tq_m|\ge\frac{q_n}{3n^{6}}.$ The summation over such $(s,t)$
$$
\ll \Big(\log\big(\frac{q_n}{3n^{6}}\big)\Big)^{-A}\sum_{\substack{1\le s,t \le n^{3} }} \left(\frac{1}{s}+\frac{1}{n^{3}}\right)\left(\frac{1}{t}+\frac{1}{n^{3}}\right)\ll \frac{(\log n)^{2}}{(n-18\log n)^{A}}.$$

(2.2) $|sq_n-tq_m|<\frac{q_n}{3n^{6}}.$
In this case, corresponding to $(m,n)$, there is at most one pair $(s,t)=(s_0,t_0)$ with $\gcd(s_0,t_0)=1$ satisfying the condition;
%
any other solution $(s,t)$ will take the form $(ks_0, kt_0)$ with $1\le k\le n^{3}.$
Since  $s_0\neq 0,$  $\frac{1}{t_0}\le \frac{1}{s_0}\big(\frac{q_m}{q_n}+\frac{1}{3n^{6}}\big).$
Hence using the trivial bound $|\widehat{\mu}(t)|\le1,$ we deduce that the corresponding summation
\begin{align*}
&\ll \sum_{k=1}^{n^{3}}\left[\min\left(\frac{1}{k}, 2\psi(q_n)\right)\min\left(\frac{1}{k}\Big(\frac{q_m}{q_n}+\frac{1}{3n^{6}}\Big),2\psi(q_m)\right)+\frac{2}{kn^{3}}+\Big(\frac{1}{n^{3}}\Big)^{2}\right]\\
&\le \sum_{k=1}^{n^{3}}\min\left(\frac{1}{k}, 2\psi(q_n)\right)\min\left(\frac{q_m}{kq_n},2\psi(q_m)\right)+\sum_{k=1}^{n^{3}}\frac{1}{3k^{2}n^{6}}+\frac{6\log n}{n^{3}}+\frac{1}{n^{3}}\\
&\le \sum_{k\le \frac{q_m}{\psi(q_m)q_n}}4\psi(q_n)\psi(q_m)+\sum_{\frac{q_m}{\psi(q_m)q_n}<k<\frac{1}{\psi(q_m)}}2\psi(q_n)\cdot\frac{q_m}{kq_n}+\sum_{k\ge \frac{1}{\psi(q_m)}}\frac{q_m}{k^{2}q_n}+\frac{12\log n}{n^{3}}\\
&\ll \psi(q_n)\cdot\frac{q_m}{q_n}+\psi(q_n)\log\frac{q_n}{q_m}\cdot\frac{q_m}{q_n}+\psi(q_m)\cdot\frac{q_m}{q_n}+\frac{12\log n}{n^{3}}\\
&\le 2\psi(q_n)\cdot\left(\frac{q_m}{q_n}\right)^{\frac{1}{2}}+\psi(q_m)\cdot\frac{q_m}{q_n}+\frac{12\log n}{n^{3}},
\end{align*}
where the second inequality is due to  the fact $\min\{a+b, c\}\le \min\{a, c\}+\min\{b,c\}$ for $a, b, c>0;$ the last one follows from $\log x\le x^{\frac{1}{2}}$ for $x\ge 1.$

since  $(q_n)_{n\in \N}$ is lacunary,
 $\sum_{1\le m<n}\big(\frac{q_m}{q_n}\big)^{\frac{1}{2}}\ll 1$. Taking this into account, we combine   \textbf{Cases 1-4} to obtain
$$\sum_{a\le m<n\le b}\mu(E_{q_n}^{\gamma}\cap E_{q_m}^{\gamma})\le 4\sum_{a\le m<n\le b}\psi(q_m)\psi(q_n)+O\Big(\sum_{n=a}^{b}\psi(q_n)\Big).$$
\subsection{Conclusion} Therefore we have that
\begin{align*}
\text{r.h.s of } (\ref{expect})  \le &\sum_{n=a}^{b}\mu(E_{q_{n}}^{\gamma})+\left(\sum_{n=a}^{b}\mu(E_{q_n}^{\gamma})\right)^{2}\\
&-4\sum_{n=a}^{b}\psi(q_n)\Big(\sum_{n=a}^{b}\mu(E_{q_n}^{\gamma})-\sum_{n=a}^{b}\psi(q_n)\Big)+O(\sum_{n=a}^{b}\psi(q_n))\\
=&O\Big(\sum_{n=a}^{b}\psi(q_n)\Big)=O\Big(\sum_{n=a}^{b}f_n\Big).
\end{align*}

Applying Lemma \ref{Philipp} with $\phi_n=f_n,$ we complete the proof.

\section{Proof of Theorem \ref{TZ2}}
In this section, we will demonstrate Theorem \ref{TZ2}. Note that Fraser-Wheeler \cite{FW23} presented a detailed Fourier dimension analysis of the set of real numbers with the partial quotient in their continued fraction expansion growing at a certain rate. In our case, we study the behavior of  the product of consecutive  partial quotients. 

We will first prove the theorem  for a specific approximating function $\Phi(q)=\frac{1}{3}q^{\tau}$ with $\tau>0$, and then  extend the result to a general approximating function. 

\subsection{An analogue of Kaufman's measure}\label{Kaufman meas}
Take $\Phi(q)=\frac{1}{3}q^{\tau}$. We are now in position to construct the probability measure $\mu$ supported on the objective set
$\mathcal{G}(3\Phi) \backslash \mathcal{K}(3\Phi)$ in Theorem \ref{TZ2}.

Let $0<\varepsilon<\min\{\frac{1}{100\tau},\frac{1}{8}\}$ be fixed. By Lemma \ref{J}, we choose $N$  so large   that $\dim_{\rm H}\textbf{Bad}(N)>1-{\varepsilon}.$
Writing
$$\Sigma_{m}:=\sum_{a_1=1}^{N}\cdots\sum_{a_m=1}^{N}q_m(a_1,\ldots,a_m)^{-2(1-\varepsilon)},$$
we have that
 $\Sigma_m\to \infty$ as $m\to\infty$; we may take $m$ sufficiently large  so that $m\ge \frac{\log N}{\varepsilon}+1$ and $\Sigma_m\ge 2^{10}.$

A coding space $\Theta^\mathbb N$ with $\Theta=\{1,\ldots,N\}$ underlines the Cantor set $\textbf{Bad}(N)$, and thus an irrational number $x$ in $\textbf{Bad}(N)$ is associated with its continued fraction expansion $[a_1(x), a_2(x),\ldots]$ as an element in the symbolic space $\Theta^\mathbb N$.
By regrouping the digits we may regard   $\Theta^\mathbb N$, or equivalently $\textbf{Bad}(N)$,  as   $(\Theta^m)^\mathbb N $ constructed by $m$-blocks in $\Theta^m :=\{1,\ldots,N\}^{m}.$

 We define a probability measure $\lambda_m$ on $\Theta^m $
by setting
$$\lambda_m\big((a_1,\ldots,a_m)\big)=\frac{q_m(a_1,\ldots,a_m)^{-2(1-\varepsilon)}}{\Sigma_m} \quad   \big((a_1,\ldots,a_m)\in \Theta^m  \big). $$
We take $\sigma_{m}$ so that $m\sigma_{m}$ is the mean of   $\log q_m(a_1,\ldots,a_m)$
with respect to this  measure, that is,
$$m\sigma_{m}=\sum_{(a_1,\ldots,a_m)\in\Theta^m }\log q_m(a_1,\ldots,a_m) \lambda_m\big((a_1,\ldots,a_m)\big).$$
By Lemma \ref{K1}(1), we have
$\log q_m(a_1,\ldots,a_m)\ge \log q_m(1,\ldots,1)\ge (m-1)\log \sqrt{2}$, and
$m\sigma_{m}\ge (m-1)\log \sqrt{2}.$

Denote
$$Y_1=q_m(a_1,\ldots,a_m),  Y_2=q_m(a_{m+1},\ldots,a_{2m}), \ldots.$$
Hence $(Y_j)_{j=1}^{\infty}$ forms a sequence of independent and identically distributed random variables
over the space
$(\Theta^m)^\mathbb N$
endowed with the probability measure
$\lambda_m\times \lambda_m\times\cdots.$
By the weak law of large numbers, there exists $j_0\ge1$ such that the set $\mathcal{E}=\mathcal{E}(j_0)\subseteq  (\Theta^m)^{j_0}$
on which
\begin{equation}\label{mean}
\left|\frac{\log Y_1+\cdots+\log Y_{j_0}}{j_0}-\mathbb{E}(\log Y_1)\right| \le \varepsilon \mathbb{E}(\log Y_1)
\end{equation}
is of measure
\begin{equation*}
\underbrace{\lambda_m\times\cdots\times \lambda_m}\limits_{j_0}(\mathcal{E})>\frac{1}{2}.
\end{equation*}

Fix such a  $j_0$ and set $p=j_0m$,
$\nu_p=
\lambda_m\times\cdots\times \lambda_m~ (j_0\text{ times})$.
 We deduce from (\ref{mean}) that, if $l\ge 1$ and $(a_1,\ldots,a_{lp})\in\mathcal{E}^{l}$, then
\begin{align*}
&\Big|\log q_{lp}(a_1,\ldots,a_{lp})-lp\sigma_m\Big|\\
\le& \Big|\log q_{lp}-\sum_{i=0}^{l-1}\log q_p(a_{ip+1},\ldots,a_{(i+1)p} )\Big|+\Big|\sum_{i=0}^{l-1}\log q_p(a_{ip+1},\ldots,a_{(i+1)p})-lp\sigma_m\Big|\\
\le& l\log 2+\sum_{i=0}^{l-1}\Big|\log q_p(a_{ip+1},\ldots,a_{(i+1)p})-\sum_{j=1}^{j_0-1}\log q_m(a_{ip+jm+1},\ldots,a_{ip+(j+1)m})\Big| \\
&+\sum_{i=0}^{l-1}\Big|\sum_{j=0}^{j_0-1}\log q_m(a_{ip+jm+1},\ldots,a_{ip+(j+1)m}-p\sigma_m\Big|\\
\le & 2lj_0\log 2+\varepsilon lp\sigma_m=\Big(\frac{2\log 2}{m\sigma_m}+\varepsilon\Big)lp\sigma_m<2\varepsilon lp\sigma_m
\end{align*}
and consequently
\begin{equation}\label{meas2}
\text{exp}\big(lp\sigma_m(1-2\varepsilon)\big) \le q_{lp}(a_1,\ldots,a_{lp})\le\text{exp}\big(lp\sigma_m(1+2\varepsilon)\big) .
\end{equation}

We write $\bar{\nu}_p$ to be the   measure induced by $\nu_p$ on $\mathcal{E},$ i.e.,
$\bar{\nu}_p(A)=\frac{\nu_p(A\cap \mathcal{E})}{\nu_p(\mathcal{E})}$, and $(\bar{\nu}_p)^{n}$ denotes the product measure $\bar{\nu}_p\times\cdots\times\bar{\nu}_p~(n\text { times}).$

Let us fix some notation:
 we always use boldface letter (e.g. $\bf{a}$) to denote a $p$-tuple in $\mathcal{E}$. And thus, ${\bf{G}}=({\bf{a}},c,4, {\bf{b})}$, for example, denotes
 a finite word consisting of  $2p+2$ digits which is the concatenation of the $p$-tuple $\bf{a}$, the digits $c$, $4$ and the tuple $\bf{b}$. Moreover, if there is no risk of confusion, we will drop the length-indicating subscript $2p+2$ in $p_{2p+2}({\bf G})$, $q_{2p+2}({\bf G})$ and $I_{2p+2}({\bf G})$.
For finite words ${\bf G, H}$, we write $({\bf G,H})$ or ${\bf G\cdot H}$
for the concatenation of ${\bf G}$ and $\bf H$.

We have the following estimates on the measure $(\bar{\nu}_p)^{n}$; a detailed proof can be found in \cite{FW23} (Lemmas 11.3 \& 11.7).

\begin{lemma}\label{L+U}
For  ${\bf{G}}=({\bf{a}}_{1},\ldots,{\bf{a}}_{n})\in \text{\rm supp}(\bar{\nu}_p)^{n},$ we have
$$|I({\bf{G}})|^{1+2\varepsilon}\le (\bar{\nu}_p)^{n}(I({\bf{G}}))\le |I({\bf{G}})|^{1-\varepsilon}.$$

\end{lemma}





\subsection{Construction of Cantor-like subset}\label{sub4.2}

Choose a rapidly increasing sequence of  integers $\{n_k\}_{k=1}^{\infty}$ such that for all $k,$
$$n_{k+1}\varepsilon \ge12k\log 2+(1+\tau)\big[k(k+1)+k(n_1+\cdots+n_{k})\big].$$

Denote by $E$ the set of
$$x=[{\bf{a}}_{1},{\bf{a}}_{2},\ldots,{\bf{a}}_{n_1},4,c_{1},{\bf{a}}_{n_1+1},{\bf{a}}_{n_1+2},
\ldots,{\bf{a}}_{n_2},4,c_{2},\ldots]$$
such that each ${\bf{a}}_j:=(a_{jp+1},\ldots,a_{(j+1)p})$ belongs to $\mathcal{E},$ and each integer  $c_k$ satisfies
$$\frac{1}{4}q({\bf{a}}_{1},\ldots,{\bf{a}}_{n_k},4)^{\tau}\le c_k\le \frac{1}{2}q({\bf{a}}_{1},\ldots,{\bf{a}}_{n_k},4)^{\tau}.$$
By the construction of $E,$ we  directly show that
$$E\subseteq \mathcal{G}(3\Phi) \backslash \mathcal{K}(3\Phi).$$

To   study the Cantor structure of the set  $E$, we define
$$\Omega_n=\left\{(a_1,\ldots,a_n)\in \N^{n}\colon    E\cap I_n(a_1,\ldots, a_n)\not=\emptyset\right\}.$$
Then
$$E=\bigcap_{n=1}^{\infty}\bigcup_{(a_1,\ldots,a_n)\in \Omega_n}I_{n}(a_1,\ldots,a_n).$$
An element $(a_1,\ldots,a_n)$ of $\Omega_n$  is called an \emph{admissible word}, 
and   the corresponding    cylinder  $I_n(a_1,\ldots, a_n)$ is called an \emph{admissible cylinder}.

\begin{lemma}\label{lem3-1}
Let ${\bf{G}}=({\bf{a}}_{1},\ldots,{\bf{a}}_{n_1}, 4, c_1, \ldots, {\bf{a}}_{n_k},4, c_k,\ldots,{\bf{a}}_n)$ be an admissible word
of length $np+2k$, where
 $n_k< n\le n_{k+1}$. Then we have
\begin{equation*}
\big|\log q({\bf{G}})-(n+\tau n_k)p\sigma_m\big|\le 4\varepsilon np\sigma_m.
\end{equation*}
\end{lemma}

\begin{proof}
Write ${\bf{G}}_1=({\bf{a}}_{1},\ldots,{\bf{a}}_{n_1}, 4, c_1, \ldots, {\bf{a}}_{n_{k-1}},4, c_{k-1},\ldots,{\bf{a}}_{n_k})$.
 We eliminate all the terms $(4,c_i)$ with $i\in\{1,\ldots,k\}$ in ${\bf{G}}$ and denote the resulted  word by ${\bf{G}}'$; we eliminate all   $(4,c_i)$ with $i\in\{1,\ldots,k-1\}$ in
${\bf{G}}_1$ to obtain  ${\bf{G}}'_1$.

 By Lemma \ref{K1}(2),  we have that
\begin{align*}
&\big|\log q({\bf{G}})-(n+\tau n_k)p\sigma_m\big|\\
\le &\big|\log q({\bf{G}}')-np\sigma_m\big|+\sum_{j=1}^{k-1}\log c_j+|\log c_k-\tau n_kp\sigma_m|+6k\log 2\\
\le &2\varepsilon np\sigma_m+(1+\tau)\sum_{j=1}^{k-1}\log c_j+\tau\big|\log q({\bf{G}}'_1)-\tau n_kp\sigma_m\big|+12k\log 2\\
\le & (\tau+2)\varepsilon np\sigma_m+(1+\tau)\sum_{j=1}^{k-1}\log c_j+12k\log 2\le 4\varepsilon np\sigma_m,
\end{align*}
where the last inequality follows by the definitions of $\{n_k\}$ and $c_j$.
\end{proof}

\subsection{Mass distribution on $E$}
We define a measure $\mu$ on $E$  via assigning mass  among the admissible  cylinders:
for $0< n\le n_1,$
 \begin{align*}
\mu(I({\bf{a}}_{1},\ldots,{\bf{a}}_{n}))
&=\bar{\nu}_p(I({\bf{a}}_{1}))\times\cdots\times\bar{\nu}_p(I({\bf{a}}_{n})),\\
\mu(I({\bf{a}}_{1},\ldots,{\bf{a}}_{n_{1}},4))
&=\mu(I({\bf{a}}_{1},\ldots,{\bf{a}}_{n_1})),\\
\mu(I({\bf{a}}_{1},\ldots,{\bf{a}}_{n_{1}},4,c_1))
&=\mu(I({\bf{a}}_{1},\ldots,{\bf{a}}_{n_{1}},4))\times \frac{4}{q({\bf{a}}_{1},\ldots,{\bf{a}}_{n_{1}},4)^{\tau}}.
\end{align*}
And inductively, for $n_k< n\le n_{k+1}$ with $k\ge 1$, we define
\begin{align*}
\mu&(I({\bf{a}}_{1},\ldots,{\bf{a}}_{n}))
=\mu(I({\bf{a}}_{1},\ldots,{\bf{a}}_{n_{k}},4,c_k))\times\bar{\nu}_{p}(I({\bf{a}}_{n_{k}+1}))\times\cdots\times \bar{\nu}_{p}(I({\bf{a}}_{n})),\\
\mu&(I({\bf{a}}_{1},\ldots,{\bf{a}}_{n_{k+1}},4))
=\mu(I({\bf{a}}_{1},\ldots,{\bf{a}}_{n_{k+1}})),\\
\mu&(I({\bf{a}}_{1},\ldots,{\bf{a}}_{n_{k+1}},4,c_{k+1}))
=\mu(I({\bf{a}}_{1},\ldots,{\bf{a}}_{n_{k+1}},4))\times \frac{4}{q({\bf{a}}_{1},\ldots,{\bf{a}}_{n_{k+1}},4)^{\tau}}.
\end{align*}

Due to the consistency property, we  extend $\mu$ to a measure, denoted still by $\mu$, to all Borel sets.

\subsection{H\"{o}lder exponent of $\mu$}

In this subsection, we derive some key geometric properties of  $\mu.$ 

We first estimate the $\mu$-measure of admissible cylinders.

\begin{lemma}\label{Holder1}
For $k\ge 1,$ we have the following estimates$\colon$

{\rm{(1)}} If ${\bf{G}}=({\bf{a}}_{1},\ldots,{\bf{a}}_{n_{k}})$ or
                  ${\bf{G}}=({\bf{a}}_{1},\ldots,{\bf{a}}_{n_{k}},4),$
then $$ |I({\bf{G}})|^{1+O(\varepsilon)}\le\mu(I({\bf{G}}))\le |I({\bf{G}})|^{1-O(\varepsilon)}.$$

{\rm{(2)}} If ${\bf{G}}=({\bf{a}}_{1},\ldots,{\bf{a}}_{n_{k}},4,c_k),$ then
 $$|I({\bf{G}})|^{\frac{\tau+2}{2\tau+2}+O(\varepsilon)}\le\mu(I({\bf{G}}))\le |I({\bf{G}})|^{\frac{\tau+2}{2\tau+2}-O(\varepsilon)}.$$

{\rm{(3)}} If ${\bf{G}}=({\bf{a}}_{1},\ldots,{\bf{a}}_{n})$ with $n_{k}<n< n_{k+1},$
then $$|I({\bf{G}})|^{1+O(\varepsilon)}\le\mu(I({\bf{G}}))\le |I({\bf{G}})|^{\frac{\tau+2}{2\tau+2}-O(\varepsilon)}.$$
 \end{lemma}
\begin{proof}
We proceed  the proof by induction.

\textsc{Case 1$\colon$}    ${\bf{G}}=({\bf{a}}_{1},\ldots,{\bf{a}}_{n_{1}}).$ By Lemma \ref{L+U}, we have
$$|I({\bf{G}})|^{1+2\varepsilon}\le\mu(I({\bf{G}}))\le |I({\bf{G}})|^{1-\varepsilon},$$
which yields that
$$|I({\bf{G}},4)|^{1+O(\varepsilon)}\le\mu(I({\bf{G}},4))\le |I({\bf{G}},4)|^{1-O(\varepsilon)}.$$

\textsc{Case 2$\colon$}  ${\bf{G}}=({\bf{a}}_{1},\ldots,{\bf{a}}_{n_{k}},4,c_k)$  with  $k\ge 1.$ Induction Hypothesis means that
$$|I({\bf{a}}_{1},\ldots,{\bf{a}}_{n_{k}},4)|^{1+O(\varepsilon)}\le\mu(I({\bf{a}}_{1},\ldots,{\bf{a}}_{n_{k}},4))\le |I({\bf{a}}_{1},\ldots,{\bf{a}}_{n_{k}},4)|^{1-O(\varepsilon)}.$$
Then by the definition of measure $\mu,$ we have
\begin{align*}
\mu(I({\bf{G}}))&=\mu(I({\bf{a}}_{1},\ldots,{\bf{a}}_{n_{k}},4))\cdot \frac{4}{q({\bf{a}}_{1},\ldots,{\bf{a}}_{n_{k}},4)^{\tau}}
\\&\le |I({\bf{a}}_{1},\ldots,{\bf{a}}_{n_{k}},4)|^{1-O(\varepsilon)}\cdot \frac{4}{q({\bf{a}}_{1},\ldots,{\bf{a}}_{n_{k}},4)^{\tau}}
\\&\le \frac{1}{q({\bf{a}}_{1},\ldots,{\bf{a}}_{n_{k}},4)^{2(1-O(\varepsilon))}}\cdot \frac{4}{q({\bf{a}}_{1},\ldots,{\bf{a}}_{n_{k}},4)^{\tau}}
\\&\le \left(\frac{1}{q({\bf{a}}_{1},\ldots,{\bf{a}}_{n_{k}},4)^{2\tau+2}}\right)^{\frac{\tau+2}{2\tau+2}-O(\varepsilon)}
\\&\le \left(\frac{1}{2q({\bf{a}}_{1},\ldots,{\bf{a}}_{n_{k}},4,c_k)^{2}}\right)^{\frac{\tau+2}{2\tau+2}-O(\varepsilon)}\le |I({\bf{G}})|^{\frac{\tau+2}{2\tau+2}-O(\varepsilon)}.
\end{align*}
Similar calculations   show that
\begin{align*}
\mu(I({\bf{G}}))&\ge |I({\bf{a}}_{1},\ldots,{\bf{a}}_{n_{k}},4)|^{1+O(\varepsilon)}\cdot \frac{1}{q({\bf{a}}_{1},{\bf{a}}_{2},\ldots,{\bf{a}}_{n_{k}},4)^{\tau}}
\\&\ge \frac{1}{q({\bf{a}}_{1},{\bf{a}}_{2},\ldots,{\bf{a}}_{n_{k}},4)^{2(1+O(\varepsilon))}}\cdot \frac{1}{q({\bf{a}}_{1},{\bf{a}}_{2},\ldots,{\bf{a}}_{n_{k}},4)^{\tau}}
\\&\ge \left(\frac{1}{q({\bf{a}}_{1},{\bf{a}}_{2},\ldots,{\bf{a}}_{n_{k}},4)^{2\tau+2}}\right)^{\frac{\tau+2}{2\tau+2}+O(\varepsilon)}
\\&\ge \left(\frac{1}{q({\bf{a}}_{1},{\bf{a}}_{2},\ldots,{\bf{a}}_{n_{k}},4,c_k)^{2}}\right)^{\frac{\tau+2}{2\tau+2}+O(\varepsilon)}\ge |I({\bf{G}})|^{\frac{\tau+2}{2\tau+2}+O(\varepsilon)}.
\end{align*}

\textsc{Case 3$\colon$}   ${\bf{G}}=({\bf{a}}_{1},\ldots,{\bf{a}}_{n_{k}},4,c_k, {\bf{a}}_{n_{k}+1},\ldots,{\bf{a}}_{n})$ with $n_k<n<n_{k+1}.$ We deduce
\begin{align*}
\mu(I({\bf{G}}))&=\mu(I({\bf{a}}_{1},\ldots,{\bf{a}}_{n_{k}},4,c_k))\bar{\nu}_p({\bf{a}}_{n_{k}+1})\times\cdots\times\bar{\nu}_p({\bf{a}}_{n})\\
&\le \left|I({\bf{a}}_{1},\ldots,{\bf{a}}_{n_{k}},4,c_k)\right|^{\frac{\tau+2}{2\tau+2}-O(\varepsilon)}
\left|I({\bf{a}}_{n_{k}+1},\ldots,{\bf{a}}_{n})\right|^{1-\varepsilon}\\
&\le\left(\left|I({\bf{a}}_{1},\ldots,{\bf{a}}_{n_{k}},4,c_k)\right|\cdot \left|I({\bf{a}}_{n_{k}+1},\ldots,{\bf{a}}_{n})\right|\right)^{\frac{\tau+2}{2\tau+2}-O(\varepsilon)}\\
&\le  |I({\bf{G}})|^{\frac{\tau+2}{2\tau+2}-O(\varepsilon)}.
\end{align*}
On the other hand, we have
\begin{align*}
\mu(I({\bf{G}}))&\ge \left|I({\bf{a}}_{1},\ldots,{\bf{a}}_{n_{k}},4,c_k)\right|^{\frac{\tau+2}{2\tau+2}+O(\varepsilon)}
\left|I({\bf{a}}_{n_{k}+1},\ldots,{\bf{a}}_{n})\right|^{1+2\varepsilon}\\
&\ge\left(\left|I({\bf{a}}_{1},\ldots,{\bf{a}}_{n_{k}},4,c_k)\right|\cdot \left|I({\bf{a}}_{n_{k}+1},\ldots,{\bf{a}}_{n})\right|\right)^{1+O(\varepsilon)}\\
&\ge  |I({\bf{G}})|^{1+O(\varepsilon)}.
\end{align*}

\textsc{Case 4$\colon$}   ${\bf{G}}=({\bf{a}}_{1},\ldots,{\bf{a}}_{n_{k}},4,c_k, {\bf{a}}_{n_{k}+1},\ldots,{\bf{a}}_{n_{k+1}}).$
  We deduce
\begin{align*}
 \ \ \ \ \  \ \ \ \ \ \mu(I({\bf{G}}))&=\mu(I({\bf{a}}_{1},\ldots,{\bf{a}}_{n_{k}},4,c_k))\bar{\nu}_p({\bf{a}}_{n_{k}+1})\times\cdots\times\bar{\nu}_p({\bf{a}}_{n_{k+1}})\\
&\le \left|I({\bf{a}}_{1},\ldots,{\bf{a}}_{n_{k}},4,c_k)\right|^{\frac{\tau+2}{2\tau+2}-O(\varepsilon)}
\left|I({\bf{a}}_{n_{k}+1},\ldots,{\bf{a}}_{n_{k+1}})\right|^{1-\varepsilon}
\\
&=(\left|I({\bf{a}}_{1},\ldots,{\bf{a}}_{n_{k}},4,c_k)\right| |I({\bf{a}}_{n_{k}+1},\ldots,{\bf{a}}_{n_{k+1}})|)^{1-O(\varepsilon)}  \left|I({\bf{a}}_{1},\ldots,{\bf{a}}_{n_{k}},4,c_k)\right|^{\frac{-\tau}{2\tau+2}}
\\
&\le  |I({\bf{G}})|^{1-O(\varepsilon)},
\end{align*}
where the last inequality holds by $|I({\bf{G}})|/|I({\bf{a}}_{1},\ldots,{\bf{a}}_{n_{k}},4,c_k)| \to 0.$  On the other hand,
\begin{align*}
\mu(I({\bf{G}}))&\ge\left|I({\bf{a}}_{1},\ldots,{\bf{a}}_{n_{k}},4,c_k)\right|^{\frac{\tau+2}{2\tau+2}+O(\varepsilon)}
\left|I({\bf{a}}_{n_{k}+1},\ldots,{\bf{a}}_{n_{k+1}})\right|^{1+\varepsilon}\\
&\ge \left(|I({\bf{a}}_{1},\ldots,{\bf{a}}_{n_{k}},4,c_k)|\cdot |I({\bf{a}}_{n_{k}+1},\ldots,{\bf{a}}_{n_{k+1}})|\right)^{1+O(\varepsilon)} \\
&\ge  |I({\bf{G}})|^{1+O(\varepsilon)}.
\end{align*}

\end{proof}

We are now in a position to establish the H\"{o}lder  exponent of the measure $\mu$ on  balls.
\begin{lemma}\label{Holder2}
Let $I$ be an interval of length $h$.  For sufficiently small $h,$ we have
$$\mu(I)\le h^{\frac{2}{\tau+2}-O(\varepsilon)}.$$
\end{lemma}
\begin{proof}
If the interval $I$ intersects the support of $\mu,$  there is a longest admissible word
${\bf{G}}$ such that $I\cap \text{supp}(\mu)\subset I({\bf{G}}).$ The
proof falls naturally into three parts according to the form of ${\bf{G}}\colon$
%
%
%

\textsc{Case 1$\colon$}  ${\bf{G}}=({\bf{a}}_{1},\ldots,{\bf{a}}_{n})$ for some $n\neq n_{k}.$

 Since ${\bf{G}}$ is the longest admissible word  such that $I\cap \text{supp}(\mu)\subset I({\bf{G}}).$ Combining this with Lemma \ref{cylinder-dis},  we have that $h$ must be at least the minimum of the distances between  $I({\bf{G}}, {\bf{a}})$ and $I({\bf{G}}, {\bf{b}})$ when ${\bf{a}}\not= {\bf{b}}$ run through the set $\text{supp} (\bar{\nu}_p).$ Further, the gap between $I({\bf{G}}, {\bf{a}})$ and $I({\bf{G}}, {\bf{b}})$ is at least an $(N, p)$-dependent constant times $|I({\bf{G}})|.$ Thus, by Lemma \ref{Holder1}, we have
$$\mu(I)\le\mu(I({\bf{G}}))\le |I({\bf{G}})|^{\frac{\tau+2}{2\tau+2}-O(\varepsilon)}\le h^{\frac{2}{\tau+2}-O(\varepsilon)}.$$

\textsc{Case 2$\colon$} ${\bf{G}}=({\bf{a}}_{1},\ldots,{\bf{a}}_{n_{k}})$ or
                  ${\bf{G}}=({\bf{a}}_{1},\ldots,{\bf{a}}_{n_{k}},4).$

Since $\mu(I({\bf{a}}_{1},\ldots,{\bf{a}}_{n_{k}}))=\mu(I({\bf{a}}_{1},\ldots,{\bf{a}}_{n_{k}},4))$
and the lengths of two cylinders $I({\bf{a}}_{1},\ldots,{\bf{a}}_{n_{k}})$ and $I({\bf{a}}_{1},\ldots,{\bf{a}}_{n_{k}},4)$ are comparable, 
we need only deal with the case ${\bf{G}}=({\bf{a}}_{1},\ldots,{\bf{a}}_{n_{k}}).$
We consider  two subcases according to the size of $h.$

\textsc{Case 2-1$\colon$}   $h>|I({\bf{G}})|^{\frac{\tau+2}{2}}.$ By Lemma
\ref{Holder1}, we obtain
$$\mu(I)\le \mu(I({\bf{G}}))\le |I({\bf{G}})|^{1-O(\varepsilon)}\le h^{\frac{2}{\tau+2}-O(\varepsilon)}.$$

\textsc{Case 2-2$\colon$}
 $h\le |I({\bf{G}})|^{\frac{\tau+2}{2}}.$
In this case, since
$$|I({\bf{a}}_{1},\ldots,{\bf{a}}_{n_{k}},4,c_k)|\ge \frac{1}{2q({\bf{a}}_{1},\ldots,{\bf{a}}_{n_{k}},4,c_k)^{2}}\ge \frac{1}{32q({\bf{G}})^{2\tau+2}},$$ there are at most
$$32hq({\bf{G}})^{2\tau+2}+2\le64hq({\bf{G}})^{2\tau+2}$$
number of cylinders of the form $({\bf{G}},4,c_k)$  intersecting $I,$ and thus
\begin{align*}
\mu(I)&\le 64hq({\bf{G}})^{2\tau+2}\mu(I({\bf{G}},4,c_k))\\
          &\le 64hq({\bf{G}})^{2\tau+2}\mu(I({\bf{G}},4))\cdot\frac{1}{q({\bf{G}},4)^{\tau}}\\
          &\le hq({\bf{G}})^{2\tau+2}\left(\frac{1}{q({\bf{G}})}\right)^{2-O(\varepsilon)}\frac{1}{q({\bf{G}})^{\tau}}\\
          &=hq({\bf{G}})^{\tau-O(\varepsilon)}\le h^{\frac{2}{\tau+2}-O(\varepsilon)},
\end{align*}
where the penultimate inequality holds by Lemma \ref{Holder1}; the last one follows by
$h\le |I({\bf{G}})|^{\frac{\tau+2}{2}} \le q({\bf{G}})^{-(\tau+2)}.$

\textsc{Case 3$\colon$} ${\bf{G}}=({\bf{a}}_{1},\ldots,{\bf{a}}_{n_{k}},4,c_k).$

Analysis similar to that in \textsc{Case 1} shows that
$$\mu(I)\le h^{\frac{2}{\tau+2}-O(\varepsilon)}.$$
\end{proof}

\subsection{Geometry of relative measures}
In this subsection, we define the relative measure $\mu_{{\bf{G}}},$ and  derive some key geometric properties of  $\mu_{{\bf{G}}}.$
Let ${\bf{G}}=(a_1,\ldots,a_n)$ be an admissible word. 
We define the relative measure $\mu_{{\bf{G}}} $ as
$$\mu_{{\bf{G}}}(I({\bf{H}}))=\frac{\mu(I({\bf{G}}\cdot{\bf{H}}))}{\mu(I({\bf{G}}))},$$
where  ${\bf{H}}$ is a finite word such that the concatenation  ${\bf{G}}\cdot{\bf{H}}$ is   admissible.

\begin{de}\label{e+t}
Let $\zeta>1$. If there exists   $k\in \N$ such that
$$(1-4\varepsilon)n_kp\sigma_{m}<\log \zeta<(\tau+1+4\varepsilon)n_kp\sigma_m,$$ we call
$\zeta$ an {\rm{exceptional  scale}}.
If
$$(\tau+1+4\varepsilon)n_kp\sigma_m\le \log \zeta\le (1-4\varepsilon)n_{k+1}p\sigma_{m}$$ for some $k\in \N$, we call $\zeta$ a {\rm{typical scale}}.
\end{de}

For a sufficiently  large $\xi$ and some $ \alpha\in(0,\frac{1}{3}),$ we write $\zeta=|\xi|^{\alpha}.$
When $\zeta$ is a typical scale, we define
\begin{equation*}
n(\zeta):=\Big\lfloor\frac{\log \zeta-\tau n_kp\sigma_m}{p\sigma_m}\Big\rfloor.
\end{equation*}

\begin{lemma}\label{typical}
For an admissible word ${\bf{G}}=({\bf{a}}_1,\ldots,{\bf{a}}_{n_{k}},4,c_k,{\bf{a}}_{n_{k}+1},\ldots,{\bf{a}}_{n(\zeta)})$, 
 we have$\colon$

{\rm{(1)}} $q({\bf{G}})\in [|\xi|^{\alpha-4\varepsilon},|\xi|^{\alpha+4\varepsilon}]  \text { and }   |I({\bf{G}})|\in[\frac{1}{2}|\xi|^{-2\alpha-8\varepsilon},|\xi|^{-2\alpha+8\varepsilon}].$

{\rm{(2)}} Let $I$ be an interval of length
$|I|=|\xi|^{-1+2\alpha-O(\varepsilon)}.$
Then
$$\mu_{{\bf{G}}}(I)\ll |I|^{\frac{\frac{2}{\tau+2}-2\alpha}{1-2\alpha}-O(\varepsilon)}.$$

\end{lemma}
\begin{proof}
{\rm{(1)}} It follows  by Lemmas \ref{lem3-1}, \ref{K2} and the definition of $n(\zeta)$.

{\rm{(2)}} Let $\pi$ be the projection  sending a finite or infinite sequence $(a_1,a_2,\ldots)$ to the corresponding continued fraction $[a_1,a_2,\ldots].$  We have
\begin{align*}
\mu_{\bf{G}}(I)&=\frac{\mu(\pi({\bf{G}}\cdot \pi^{-1}(I)))}{\mu(I({\bf{G}}))}\le |I({\bf{G}})|^{-1-O(\varepsilon)}\mu(\pi({\bf{G}}\cdot\pi^{-1}(I)))\\&\le |\xi|^{2\alpha+O(\varepsilon)}\mu(\pi({\bf{G}}\cdot \pi^{-1}(I))),
\end{align*}
where the first inequality holds by Lemma \ref{Holder1}, the last one follows by Lemma \ref{typical}(1).

By Lemma \ref{Holder2}, we deduce that
\begin{align*}
\mu(\pi({\bf{G}}\cdot\pi^{-1}(I)))&\le |\pi\big({\bf{G}}\cdot \pi^{-1}(I)\big)|^{\frac{2}{\tau+2}-O(\varepsilon)}\\
&\le \sup_{x_1,x_2\in I}\left|\frac{p({\bf{G}})x_1+p'({\bf{G}})}{q({\bf{G}})x_1+q'({\bf{G}})}-\frac{p({\bf{G}})x_2+p'({\bf{G}})}{q({\bf{G}})x_2+q'({\bf{G}})}\right|^{\frac{2}{\tau+2}-O(\varepsilon)}\\
&\le \left|N^{2}q'({\bf{G}})^{-2}|I|\right|^{\frac{2}{\tau+2}-O(\varepsilon)}\\
&\le \left(|\xi|^{-2\alpha+O(\varepsilon)}|\xi|^{-1+2\alpha-O(\varepsilon)}\right)^{\frac{2}{\tau+2}-O(\varepsilon)}\le |\xi|^{-\frac{2}{\tau+2}+O(\varepsilon)},
\end{align*}
here and hereafter, $q'(G)$ denotes the penultimate dominator, that is, $q'({\bf{G}})=q_{n-1}(a_1,\ldots,a_{n-1})$ if
 ${\bf{G}}=(a_1,\ldots,a_n)$; and similarly for $p'$.
Hence, we have
$$\mu_{\bf{G}}(I)\le |\xi|^{2\alpha+O(\varepsilon)}\cdot |\xi|^{-\frac{2}{\tau+2}+O(\varepsilon)}\le |\xi|^{2\alpha-\frac{2}{\tau+2}+O(\varepsilon)}\le |I|^{\frac{\frac{2}{\tau+2}-2\alpha}{1-2\alpha}-O(\varepsilon)}.$$
\end{proof}

\begin{lemma}\label{exceptional}
Fix a sufficiently  large $\xi.$ Suppose that $\xi^{\alpha}$ is an exceptional scale satisfying $(1-4\varepsilon)n_kp\sigma_m<\log |\xi|^{\alpha}<(\tau+1+4\varepsilon)n_kp\sigma_m.$ Write
$\alpha'=(\tau+1+10\varepsilon)\alpha$ and $\zeta=|\xi|^{\alpha'}.$ Then $\zeta$ is a typical scale. Furthermore, let $I$ be an interval of length 
$|I|=|\xi|^{-1+2\alpha'-O(\varepsilon)}.$
Then for any admissible word ${\bf{G}}=({\bf{a}}_1,\ldots,{\bf{a}}_{n_{k}},c_k,{\bf{a}}_{n_{k}+1},\ldots,{\bf{a}}_{n(\zeta)}),$
we have
$$\mu_{{\bf{G}}}(I)\le |I|^{1-O(\varepsilon)}.$$

\end{lemma}
\begin{proof}
We  readily check that
$|\xi|^{\alpha'}$ is a typical scale.
Write $\zeta_1=\xi^{\frac{1-2\alpha'}{2}}$ and ${\tilde n}(\zeta_1)=\lfloor\frac{\log \zeta_1}{p\sigma_m}\rfloor.$ Set
$$\Omega:=\left\{{\bf{H}}=({\bf{b}}_1,\ldots,{\bf{b}}_{{\tilde n}(\zeta_1)})\in (\text{supp}(\bar{\nu}_p))^{{\tilde n}(\zeta_1)}\colon I({\bf{G}}\cdot{\bf{H}})\cap \pi({\bf{G}}\cdot\pi^{-1}(I))\neq \emptyset \right\}.$$
Since $n_k< \varepsilon n_{k+1}$, we have $n(\zeta)+{\tilde n}(\zeta_1)<\frac{\log\zeta+\log\zeta_1}{p\sigma_m}<n_{k+1}.$ Writing $\sharp \Omega$ for the cardinality of $\Omega$, we have
\begin{align*}
\mu_{\bf{G}}(I)\le \sharp \Omega \cdot\frac{\mu (I({\bf{G}}\cdot{\bf{H}}))}{\mu(I({\bf{G}}))}=\sharp \Omega\cdot (\bar{\nu}_p)^{{\tilde n}(\zeta_1)}(I({\bf{H}})).
\end{align*}
By (\ref{meas2}), for any ${\bf{H}}\in \Omega,$ we have
$$|I({\bf{H}})|\le \zeta_1^{-2+O(\varepsilon)}=|\xi|^{-1+2\alpha'+O(\varepsilon)}\le |I|^{1-O(\varepsilon)},$$
$$|I({\bf{H}})|\ge \zeta_1^{-2-O(\varepsilon)}=|\xi|^{-1+2\alpha'-O(\varepsilon)}\ge |I|^{1+O(\varepsilon)}.$$
Thus
$$\sharp \Omega \le \frac{|I|}{ |I|^{1+O(\varepsilon)}}\le|I|^{-O(\varepsilon)}. $$
Combining with Lemma \ref{L+U}, we have
$$\mu_{\bf{G}}(I)\le |I|^{-O(\varepsilon)}\cdot |I({\bf{H}})|^{1-\varepsilon}\le |I|^{1-O(\varepsilon)}.$$
\end{proof}

\subsection{Establishing Theorem \ref{TZ2} for $\Phi(q)=\frac{1}{3}q^{\tau}$.}

In this subsection, we deal with the asymptotic behavior of $\widehat{\mu}(\xi)$
for large $|\xi|.$ Without loss of generality, we may assume that $\xi$ is positive.

Suppose that $\zeta=\xi^{\alpha}$ is a typical scale. 
Set
$$S(\zeta):=\left\{{\bf{G}}=({\bf{a}}_1,\ldots,{\bf{a}}_{n_{k}},4,c_k,{\bf{a}}_{n_{k}+1},\ldots,{\bf{a}}_{n(\zeta)})\colon {\bf{G}} \text{ is an admissible word}\right\}.$$
Given ${\bf{G}}\in S(\zeta),$ we remark that 
$$t=[{\bf{G}},x]=\frac{p({\bf{G}})x+p'({\bf{G}})}{q({\bf{G}})x+q'({\bf{G}})}.$$
And thus
\begin{equation}\label{decay}
\widehat{\mu}(\xi)=\sum_{{\bf{G}}\in S(\zeta)}\mu(I({\bf{G}}))\int e\left(-\xi \frac{p({\bf{G}})x+p'({\bf{G}})}{q({\bf{G}})x+q'({\bf{G}})}\right) \d\mu_{{\bf{G}}}(x).
\end{equation}

%

 For  ${\bf{G}}\in S(\zeta),$ by Lemma \ref{typical}(1), we know
$$ q'({\bf{G}})\in [\xi^{\alpha-5\varepsilon},\xi^{\alpha+5\varepsilon}],~
q({\bf{G}})\in [\xi^{\alpha-5\varepsilon},\xi^{\alpha+5\varepsilon}].$$
Now we cover the square $\Box=[\xi^{\alpha-5\varepsilon},\xi^{\alpha+5\varepsilon}]^2$ by a mesh of width $\xi^{\alpha-200\varepsilon}$; there are at most $\xi^{410\varepsilon}$ small squares which together cover $\Box$. We collect all left-lower endpoints of those squares to form a set $P.$ For each $a\in P,$ we set
 $$\Theta_{a}=\{{\bf{G}}\in S(\zeta)\colon  q({\bf{G}})\in [a, a+\xi^{\alpha-200\varepsilon}]\}.$$
 If $\Theta_{a}$ is nonempty, we pick a representative element ${\bf{G}}_a\in\Theta_{a}$ and write $\mu_{a}$ for $\mu_{{\bf{G}}_a}.$ We will approximate $\d\mu$ by $\d\mu_{a}.$

The sum in (\ref{decay}) can be  split 
into sums over the classes $\Theta_{a}$, that is,
\begin{align*}
\widehat{\mu}(\xi)
=&\sum_{a\in P}\sum_{{\bf{G}}\in \Theta_{a}}\mu(I({\bf{G}}))\int e\big(-\xi\frac{p({\bf{G}})x+p'({\bf{G}})}{q({\bf{G}})x+q'({\bf{G}})}\big) \d\mu_{\bf{G}}(x)\\
=&\sum_{a\in P}\int \sum_{{\bf{G}}\in \Theta_{a}}\mu(I({\bf{G}}))e\big(-\xi\frac{p({\bf{G}})x+p'({\bf{G}})}{q({\bf{G}})x+q'({\bf{G}})}\big) \d\mu_{a}(x)+\\
&\sum_{a\in P}\sum_{{\bf{G}}\in \Theta_{a}}\mu(I({\bf{G}}))\Big(\int e\big(-\xi\frac{p({\bf{G}})x+p'({\bf{G}})}{q({\bf{G}})x+q'({\bf{G}})}\big) (\d\mu_{\bf{G}}(x)-\d\mu_{a}(x))\Big)\\
:=& S_1+S_2.
\end{align*}

 \subsubsection{\textbf{\textsc{Estimation of} $S_1$}}
We use the comparison Lemma \ref{Kaufman3} to deduce an upper bound for
$$S_1=\sum_{a\in P}\int \sum_{{\bf{G}}\in \Theta_{a}}\mu(I({\bf{G}}))e\left(-\xi\frac{p({\bf{G}})x+p'({\bf{G}})}{q({\bf{G}})x+q'({\bf{G}})}\right) \d\mu_{a}(x).$$
Write
 $$F(x)=\sum_{{\bf{G}}\in \Theta_{a}}\mu(I({\bf{G}}))e\left(-\xi\frac{p({\bf{G}})x+p'({\bf{G}})}{q({\bf{G}})x+q'({\bf{G}})}\right) .$$

 \begin{lemma}\label{Mm_2}
 We have$\colon$

 {\rm{(1)}} $\max_{x\in[0,1)}\left|F'(x)\right|\le \xi^{1-2\alpha+11\varepsilon}:=M.$
 \smallskip

 {\rm{(2)}} $m_2:=\int_{0}^{1}|F(x)|^{2}\d x\le \xi^{\frac{3\alpha-1}{2}+O(\varepsilon)}+\xi^{\frac{-\alpha(\tau+2)}{\tau+1}+O(\varepsilon)}.$
 \end{lemma}
\begin{proof}
 {\rm{(1)}}   We directly calculate
$$\max_{x\in[0,1)}\left|F'(x)\right|\le\left|2\pi\xi \sum_{{\bf{G}}\in \Theta_{a}}\mu(I({\bf{G}}))\frac{1}{(q({\bf{G}})x+q'({\bf{G}}))^{2}}\right|\le \xi^{1-2\alpha+11\varepsilon}.$$

 {\rm{(2)}}  We have 
 \begin{align*}
m_2=\sum_{{\bf{G}}\in \Theta_{a}}\sum_{{\bf{G}}_1\in \Theta_{a}}\mu(I({\bf{G}}))\mu(I({\bf{G}}_1))
\int_{0}^{1}e\left(f(x)\right)\d x,
\end{align*}
where
 $$f(x)=-\xi\left(\frac{p({\bf{G}})x+p'({\bf{G}})}{q({\bf{G}})x+q'({\bf{G}})}-\frac{p({\bf{G}}_1)x+p'({\bf{G}}_1)}{q({\bf{G}}_1)x+q'({\bf{G}}_1)}\right).$$
The argument $f(x)$ admits a derivative, up to a multiplicative factor of $\pm 1,$ that
$$f'(x)= \xi\frac{(q({\bf{G}})+q({\bf{G}}_1))x+q'({\bf{G}})+q'({\bf{G}}_1)}{(q({\bf{G}})x+q'({\bf{G}}))^{2}(q({\bf{G}}_1)x+q'({\bf{G}}_1))^{2}}\big((q({\bf{G}})-q({\bf{G}}_1))x+q'({\bf{G}})-q'({\bf{G}}_1)\big),$$
which may be written as $g(x)(C_1 x+C_2)$ with 
$$g(x)=\xi\frac{(q({\bf{G}})+q({\bf{G}}_1))x+q'({\bf{G}})+q'({\bf{G}}_1)}{(q({\bf{G}})x+q'({\bf{G}}))^{2}(q({\bf{G}}_1)x+q'({\bf{G}}_1))^{2}},$$
$$C_1=q({\bf{G}})-q({\bf{G}}_1), \quad  C_2=q'({\bf{G}})-q'({\bf{G}}_1).$$
We continue to estimate
$\int_{0}^{1}e\left(f(x)\right)\d x$
by discussing whether or not there exists a stationary point of $f$. 
\begin{itemize}

\item 
$q({\bf{G}})=q({\bf{G}}_1)$ and $q'({\bf{G}})=q'({\bf{G}}_1).$
    
    It means that ${\bf{G}}={\bf{G}}_1.$
Furthermore, we have
\begin{align*}
m_2\le \sum_{{\bf{G}}\in \Theta_{a}} \sum_{{\bf{G}}_1={\bf{G}}}\mu(I({\bf{G}}))\mu(I({\bf{G}}_1))
\le |I({\bf{G}})|^{\frac{\tau+2}{2\tau+2}-O(\varepsilon)}\le \xi^{\frac{-\alpha(\tau+2)}{\tau+1}+O(\varepsilon)},
\end{align*}
where the second inequality holds by Lemma \ref{Holder1}.
\smallskip

\item 
 $q({\bf{G}})=q({\bf{G}}_1)$ but $q'({\bf{G}}) \neq q'({\bf{G}}_1).$ 
 
 The phase is non-stationary, and we have 
$$|f'(x)|\ge \xi^{1-3\alpha-O(\varepsilon)}|q'({\bf{G}})-q'({\bf{G}}_1)|,$$
$$|f''(x)|\le \xi^{1-3\alpha+O(\varepsilon)}|q'({\bf{G}})-q'({\bf{G}}_1)|.$$
Applying Lemma \ref{Kaufman1}, we obtain
$$m_2\le \sum_{{\bf{G}}\in \Theta_{a}}\sum_{{\bf{G}}_1\in \Theta_{a}}\mu(I({\bf{G}}))\mu(I({\bf{G}}_1))\frac{\xi^{-1+3\alpha+O(\varepsilon)}}{|q'({\bf{G}})-q'({\bf{G}}_1)|}\le \xi^{-1+3\alpha+O(\varepsilon)}.$$

\item 
 $q({\bf{G}})\neq q({\bf{G}}_1)$ but $q'({\bf{G}}) \neq q'({\bf{G}}_1).$ 
 
 In this case, it is easy to verify that
$|g(x)|\ge \xi^{1-3\alpha-O(\varepsilon)}$ and $|g'(x)|\le \xi^{1-3\alpha+O(\varepsilon)}.$ By Lemma \ref{Kaufman2}, we have
$$m_2\le \sum_{{\bf{G}}\in \Theta_{a}}\sum_{{\bf{G}}_1\in \Theta_{a}}\mu(I({\bf{G}}))\mu(I({\bf{G}}_1)) \xi^{\frac{-1+3\alpha+O(\varepsilon)}{2}} |q({\bf{G}})-q({\bf{G}}_1)|^{-\frac{1}{2}}
\le \xi^{\frac{-1+3\alpha+O(\varepsilon)}{2}}.$$
\end{itemize}

Combining these estimates, we complete the proof.
\end{proof}

\begin{lemma}\label{S_1} Let $\alpha_0=\frac{116-13\sqrt{73}}{144}\in(0,\frac{1}{3}).$
 If $\xi^{\alpha_0}$ is a typical scale, put $\alpha=\alpha_0$; otherwise, put $\alpha=\alpha_0(\tau+1+10\varepsilon).$ We have
$$S_1= O(\xi^{-\varepsilon}).$$
\end{lemma}
\begin{proof}
Choose $r=\xi^{-411\varepsilon}.$ By  Lemma \ref{Mm_2}(1),
$\frac{r}{M}\le \xi^{-1+2\alpha-O(\varepsilon)}.$ We consider two cases:

\noindent\textsc{Case 1$\colon$}  $\xi^{\alpha_0}$ is a typical scale.

By Lemmas \ref{Kaufman3}, \ref{typical} and   \ref{Mm_2}, we obtain \begin{align*}
S_1&\le \sum_{a\in P} \left[2\xi^{-411\varepsilon}+\xi^{-\big(\frac{2}{\tau+2}-2\alpha-O(\varepsilon)\big)}\Big(1+\big(\xi^{\frac{3\alpha-1}{2}+O(\varepsilon)}+\xi^{\frac{-\alpha(\tau+2)}{\tau+1}+O(\varepsilon)}\big)\xi^{1-2\alpha+O(\varepsilon)}\Big)\right]\\
&\le 2\xi^{-\varepsilon}+\xi^{-\big(\frac{2}{\tau+2}-2\alpha-O(\varepsilon)\big)}+\xi^{\frac{3\alpha-1}{2}+\frac{\tau}{\tau+2}+O(\varepsilon)}+\xi^{\frac{\tau}{\tau+2}-\frac{\alpha(\tau+2)}{\tau+1}+O(\varepsilon)}.
\end{align*}
Therefore we will establish the desired result if the following inequalities hold:
$$\frac{2}{\tau+2}-2\alpha>0; \ \ \frac{3\alpha-1}{2}+\frac{\tau}{\tau+2}<0; \ \ \frac{\tau}{\tau+2}-\frac{\alpha(\tau+2)}{\tau+1}<0.$$
Solving those inequalities for $\alpha$ yields
\begin{equation}\label{choose}
\frac{\tau^{2}+\tau}{(\tau+2)^{2}}<\alpha<\frac{2-\tau}{3(\tau+2)}.
\end{equation}
Observe that the left side of the inequality (\ref{choose}) is a increasing function of $\tau,$
and the right side is  decreasing. Thus, to verify that
(\ref{choose}) holds when $\tau< \frac{\sqrt{73}-3}{8}:=\tau_1$ and $\alpha=\alpha_0,$
it suffices to  check   that
$$\alpha_0=\frac{\tau^{2}_1+\tau_1}{(\tau_1+2)^{2}}=\frac{2-\tau_1}{3(\tau_1+2)}.$$

\noindent\textsc{Case 2$\colon$}  $\xi^{\alpha_0}$ is an exceptional scale.

 Let $\alpha=\alpha_0(\tau+1+10\varepsilon).$ It is easy to check that $\alpha\in(0,\frac{1}{3})$, and
$\xi^{\alpha}$ is a typical scale. We use the same analysis as \textsc{Case 1}. By Lemmas \ref{Kaufman3}, \ref{exceptional} and \ref{Mm_2}, we have
\begin{align*}
S_1&\le \sum_{a\in P} \left[2\xi^{-411\varepsilon}+\xi^{-\big(1-2\alpha-O(\varepsilon)\big)}\Big(1+\big(\xi^{\frac{3\alpha-1}{2}+O(\varepsilon)}+\xi^{\frac{-\alpha(\tau+2)}{\tau+1}+O(\varepsilon)}\big)\xi^{1-2\alpha+O(\varepsilon)}\Big)\right]\\
&\le 2\xi^{-\varepsilon}+\xi^{-\big(1-2\alpha-O(\varepsilon)\big)}+\xi^{\frac{(1-2\alpha)(3\alpha-1)}{2}+O(\varepsilon)}+\xi^{2\alpha-1-\frac{\alpha(\tau+2)}{\tau+1}+O(\varepsilon)}.
\end{align*}
And it remains to be proven that
$$1-2\alpha>0; \ \ \frac{(1-2\alpha)(3\alpha-1)}{2}<0; \ \ 2\alpha-1-\frac{\alpha(\tau+2)}{\tau+1}<0.$$
\end{proof}

 \subsubsection{\textbf{\textsc{Estimation of} $S_2$}}

\begin{lemma}\label{S_2}
Let $\alpha_0=\frac{116-13\sqrt{73}}{144}\in(0,\frac{1}{3}).$  If $\xi^{\alpha_0}$ is a typical scale, put $\alpha=\alpha_0$; otherwise, put $\alpha=\alpha_0(\tau+1+10\varepsilon).$ We have
$$S_2= O\left(\xi^{-\varepsilon}\right).$$
\end{lemma}

\begin{proof}
For $G\in \Theta_{a},$ the elements of supp$(\mu_G)$ are of the form$\colon$
$$({\bf{a}}_{n(\zeta)+1},\ldots,4,c_{k},{\bf{a}}_{n_k+1},\ldots,{\bf{a}}_{n_{k+1}},4,c_{k+1},\ldots),$$
the elements of supp($\mu_{a}$) take the form$\colon$
$$({\bf{a}}_{n(\zeta)+1},\ldots,4,c'_{k},{\bf{a}}_{n_k+1},\ldots,{\bf{a}}_{n_{k+1}},4,c'_{k+1},\ldots).$$
Set
\begin{align*}
\Omega_1^{\ast}&=\big\{{\bf{H}}:=({\bf{a}}_{n(\zeta)+1},\ldots,4,c_{k},{\bf{a}}_{n_k+1},\ldots,{\bf{a}}_{n_{k+1}})\colon {\bf{G}}\cdot {\bf{H}} \text{ is an admissible word}\big\},\\
\Omega_2^{\ast}&=\big\{{\bf{H}'}:=({\bf{a}}_{n(\zeta)+1},\ldots,4,c'_{k},{\bf{a}}_{n_k+1},\ldots,{\bf{a}}_{n_{k+1}})\colon {\bf{G}}_{a}\cdot {\bf{H}}' \text{ is an admissible word}\big\}.
\end{align*}
Then
\begin{align*}
&\int e\left(-\xi\frac{p({\bf{G}})x+p'({\bf{G}})}{q({\bf{G}})x+q'({\bf{G}})}\right) \big(\d\mu_{\bf{G}}(x)-\d\mu_{a}(x)\big)
\\=&\sum_{{\bf{H}}\in \Omega_1^{\ast}}\int_{I({\bf{H}})}e\left(-\xi\frac{p({\bf{G}})x+p'({\bf{G}})}{q({\bf{G}})x+q'({\bf{G}})}\right) \d\mu_{\bf{G}}(x)
-\sum_{{\bf{H}}'\in \Omega_2^{\ast}}\int_{I({\bf{H}}')} e\left(-\xi\frac{p({\bf{G}})x+p'({\bf{G}})}{q({\bf{G}})x+q'({\bf{G}})}\right) \d\mu_{a}(x)\\
=&S_{21}+S_{22}+S_{23},
\end{align*}
where
\begin{align*}
S_{21}:=&\sum_{{\bf{H}}\in \Omega_1^{\ast}}\int_{I({\bf{H}})}e\left(-\xi\frac{p({\bf{G}})x+p'({\bf{G}})}{q({\bf{G}})x+q'({\bf{G}})}\right) -e\left(-\xi\frac{p({\bf{G}})\frac{p({\bf{H}})}{q({\bf{H}})}+p'({\bf{G}})}{q({\bf{G}})\frac{p({\bf{H}})}{q({\bf{H}})}+q'({\bf{G}})}\right)\d\mu_{\bf{G}} (x);\\
S_{22}:=&\sum_{{\bf{H}'}\in \Omega_2^{\ast}}\int_{I({\bf{H}}')}e\left(-\xi\frac{p({\bf{G}})x+p'({\bf{G}})}{q({\bf{G}})x+q'({\bf{G}})}\right) -e\left(-\xi\frac{p({\bf{G}})\frac{p({\bf{H}'})}{q({\bf{H}'})}+p'({\bf{G}})}{q({\bf{G}})\frac{p({\bf{H}'})}{q({\bf{H}'})}+q'({\bf{G}})}\right)\d\mu_{a} (x);\\
S_{23}:=&\sum_{{\bf{H}}\in \Omega_1^{\ast}}\mu_{\bf{G}}(I({\bf{H}}))e\left(-\xi\frac{p({\bf{G}})\frac{p({\bf{H}})}{q({\bf{H}})}+p'({\bf{G}})}{q({\bf{G}})\frac{p({\bf{H}})}{q({\bf{H}})}+q'({\bf{G}})}\right)-\sum_{{\bf{H}'}\in \Omega_2^{\ast}}\mu_{a}(I({\bf{H}}))e\left(-\xi\frac{p({\bf{G}})\frac{p({\bf{H}}')}{q({\bf{H}}')}+p'({\bf{G}})}{q({\bf{G}})\frac{p({\bf{H}}')}{q({\bf{H}}')}+q'({\bf{G}})}\right).
\end{align*}

(1) Estimate of $S_{21}.$ We deduce
\begin{align*}
|S_{21}|&\le \sum_{{\bf{H}}\in \Omega_1^{\ast}} \mu_{\bf{G}} (I({\bf{H}}))\sup_{x,y\in I({\bf{H}})}\left|e\left(-\xi\frac{p({\bf{G}})x+p'({\bf{G}})}{q({\bf{G}})x+q'({\bf{G}})}\right) -e\left(-\xi\frac{p({\bf{G}})y+p'({\bf{G}})}{q({\bf{G}})y+q'({\bf{G}})}\right)\right|\\
&\le \sum_{{\bf{H}}\in \Omega_1^{\ast}} \mu_{\bf{G}} (I({\bf{H}}))\xi\sup_{x,y\in I({\bf{H}})}\left|\frac{x-y}{(q({\bf{G}})x+q'({\bf{G}}))(q({\bf{G}})y+q'({\bf{G}}))}\right|\\
&\le\sum_{{\bf{H}}\in \Omega_1^{\ast}} \mu_{\bf{G}} (I({\bf{H}}))\xi|I({\bf{H}})|.
\end{align*}
Since $n_{k+1}$ may be  chosen to be sufficiently large compared to $n_k$ such that
$|I({\bf{H}})|<\xi^{-100},$ we have $$|S_{21}|\le \xi^{-99}.$$

(2) Estimate of $S_{22}.$ Similar arguments to those above show that
$$|S_{22}|\le \xi^{-99}.$$

(3) Estimate of $S_{23}.$ We  deduce
\begin{align*}
|S_{23}|=&\left|\sum_{{\bf{H}}\backslash{\bf{H}'}\in \Omega_1^{\ast}}\mu_{\bf{G}}(I({\bf{H}}))e\left(\xi\frac{p({\bf{G}})\frac{p({\bf{H}})}{q({\bf{H}})}+p'({\bf{G}})}{q({\bf{G}})\frac{p({\bf{H}})}{q({\bf{H}})}+q'({\bf{G}})}\right)-\sum_{{\bf{H}'\backslash{\bf{H}}}\in \Omega_2^{\ast}}\mu_{a}({\bf{H}})e\left(\xi\frac{p({\bf{G}})\frac{p({\bf{H}}')}{q({\bf{H}}')}+p'({\bf{G}})}{q({\bf{G}})\frac{p({\bf{H}}')}{q({\bf{H}}')}+q'({\bf{G}})}\right)\right|\\
\le&\sum_{{\bf{H}}\backslash{\bf{H}'}\in \Omega_1^{\ast}}\mu_{\bf{G}}(I({\bf{H}}))+\sum_{{\bf{H}'}\backslash{\bf{H}}\in \Omega_2^{\ast}}\mu_{a}(I({\bf{H}'}))\\
\le&\sum_{{\bf{a}}_{n(\zeta)+1},\ldots,{\bf{a}}_{n_k}}\sum_{\frac{1}{2}q({\bf{G}}_{a},{\bf{a}}_{n(\zeta)+1},\ldots,{\bf{a}}_{n_k},4)^{\tau}\le c_{k}<\frac{1}{2}q({\bf{G}},{\bf{a}}_{n(\zeta)+1},\ldots,{\bf{a}}_{n_k},4)^{\tau}}\mu_{\bf{G}}(I({\bf{G}},{\bf{a}}_{n(\zeta)+1},\ldots,{\bf{a}}_{n_k},4,c_k))\\&+\sum_{{\bf{a}}_{n(\zeta)+1},\ldots,{\bf{a}}_{n_k}}\sum_{\frac{1}{4}q({\bf{G}}_{a},{\bf{a}}_{n(\zeta)+1},\ldots,{\bf{a}}_{n_k},4)^{\tau}\le c_{k}'<\frac{1}{4}q({\bf{G}},{\bf{a}}_{n(\zeta)+1},\ldots,{\bf{a}}_{n_k},4)^{\tau}}\mu_{a}(I({\bf{G}_{a}},{\bf{a}}_{n(\zeta)+1},\ldots,{\bf{a}}_{n_k},4,c_k'))
\\:=&T_{1}+T_{2}.
\end{align*}
We estimate $T_1$
\begin{align*}
\le&2\sum_{{\bf{a}}_{n(\zeta)+1},\ldots,{\bf{a}}_{n_k}}\frac{\mu(I({\bf{G}}, {\bf{a}}_{n(\zeta)+1},\ldots,{\bf{a}}_{n_k}))}{\mu(I({\bf{G}}))}\times\frac{q({\bf{G}}, {\bf{a}}_{n(\zeta)+1},\ldots,{\bf{a}}_{n_k},4)^{\tau}-q({\bf{G}}_{a}, {\bf{a}}_{n(\zeta)+1},\ldots,{\bf{a}}_{n_k},4)^{\tau}}{q({\bf{G}}, {\bf{a}}_{n(\zeta)+1},\ldots,{\bf{a}}_{n_k},4)^{\tau}}\\
\le&2\sum_{{\bf{a}}_{n(\zeta)+1},\ldots,{\bf{a}}_{n_k}}\frac{\mu(I({\bf{G}}, {\bf{a}}_{n(\zeta)+1},\ldots,{\bf{a}}_{n_k}))}{\mu(I({\bf{G}}))}\times\frac{q({\bf{G}}, {\bf{a}}_{n(\zeta)+1},\ldots,{\bf{a}}_{n_k})-q({\bf{G}}_{a}, {\bf{a}}_{n(\zeta)+1},\ldots,{\bf{a}}_{n_k})}{q({\bf{G}}, {\bf{a}}_{n(\zeta)+1},\ldots,{\bf{a}}_{n_k})}\\
\le& 2\sum_{{\bf{a}}_{n(\zeta)+1},\ldots,{\bf{a}}_{n_k}}\frac{\mu(I({\bf{G}}, {\bf{a}}_{n(\zeta)+1},\ldots,{\bf{a}}_{n_k}))}{\mu(I({\bf{G}}))}\times\Big(\frac{q({\bf{G}})-q({\bf{G}}_{a})}{q({\bf{G}})}+\frac{|q'({\bf{G}})-q'({\bf{G}}_{a})|}{q'({\bf{G}})}\Big)
\\ \le & O(\xi^{-\varepsilon}),\end{align*}
where the penultimate inequality follows by  
$$q(a_1,\ldots,a_n,b_1,\ldots,b_m)=q(b_1,\ldots,b_m)q(a_1,\ldots,a_{n-1})+p(b_1,\ldots,b_m)q(a_1,\ldots,a_{n}).$$

In a similar way, we show that $T_{2}\le O(\xi^{-\varepsilon}).$

\medskip

Estimates (1), (2) and (3) together yields that
$$S_2\le \sum_{a\in P}\sum_{{\bf{G}}\in \Theta_{a}}\mu(I({\bf{G}})) O(\xi^{-\varepsilon})= O(\xi^{-\varepsilon}).$$
\end{proof}

To sum up, we complete the proof of Theorem \ref{TZ2} for  $\Phi(q)=\frac{1}{3}q^{\tau}$.

\subsection{Establishing Theorem \ref{TZ2} for general function $\Phi$}\label{general}

For  $0<\varepsilon<\min\{\frac{1}{1000\tau},\frac{1}{8}\},$  we choose a rapidly increasing sequence $\{Q_k\}_{k=1}^{\infty}$ of positive integers satisfying
$$\frac{\log Q_k}{\log Q_{k+1}}<\frac{\varepsilon}{1000}, \ \ Q_k^{\tau-\varepsilon}\le 3\Phi(Q_k)\le Q_k^{\tau+\varepsilon} \ \ \text{ for all } k\ge 1,$$
where
$$\tau=\lim_{k\to\infty}\frac{\log \Phi(Q_k)}{\log Q_k}.$$

Let $n_0=0.$ We define $\{n_k\}_{k=1}^{\infty}$ as follows$\colon$
$$n_k=\max\left\{N>n_{k-1}\colon \text{exp}(Np\sigma_m(1+3\varepsilon))\le Q_k\right\}.$$
It follows for all  $k\ge 1$ that, 
$$\text{exp}\big(n_kp\sigma_m(1+3\varepsilon)\big)\le Q_k<\text{exp}\big((n_k+1)p\sigma_m(1+3\varepsilon)\big).$$

We use the same mechanics  as the one in subsection \ref{sub4.2}.  Denote by $E(\Phi)$ the set of
$$x=[{\bf{a}}_{1},{\bf{a}}_{2},\ldots,{\bf{a}}_{n_1},4,c_{1},{\bf{a}}_{n_1+1},{\bf{a}}_{n_1+2},
\ldots,{\bf{a}}_{n_2},4,c_{2},\ldots]$$
such that each ${\bf{a}}_j:=(a_{jp},\ldots,a_{(j+1)p-1})$ belongs to $\mathcal{E}$ (as defined in subsection \ref{Kaufman meas}), and each integer $c_k$ satisfies
$$\frac{1}{4}q({\bf{a}}_{1},\ldots,{\bf{a}}_{n_k},4)^{\tau+20\varepsilon}\le c_k\le \frac{1}{2}q({\bf{a}}_{1},\ldots,{\bf{a}}_{n_k},4)^{\tau+20\varepsilon}.$$
Now we check that
$$E(\Phi)\subset \mathcal{G}(3\Phi) \backslash \mathcal{K}(3\Phi).$$
For $x\in E(\Phi),$ using (\ref{meas2}) and the monotonicity of $\Phi$,  we deduce that 
\begin{align*}
4c_k&\ge q({\bf{a}}_{1},\ldots,{\bf{a}}_{n_k},4)^{\tau+20\varepsilon}
        \ge q({\bf{a}}_{1},\ldots,{\bf{a}}_{n_k})^{\tau+20\varepsilon}\\
        & \ge \text{exp}\big((\tau+20\varepsilon)n_kp\sigma_m(1-2\varepsilon)\big)
        \ge  \text{exp}\big((\tau+\varepsilon)(n_k+1)p\sigma_m(1+3\varepsilon)\big)\\
        &\ge Q_k^{\tau+\varepsilon}\ge 3\Phi(Q_k)\ge 3\Phi(\text{exp}\big(n_kp\sigma_m(1+3\varepsilon)\big))\\
        &\ge 3\Phi(q({\bf{a}}_{1},\ldots,{\bf{a}}_{n_k},4)).
\end{align*}
Hence $x\in \mathcal{G}(3\Phi) \backslash \mathcal{K}(3\Phi).$


The Fourier dimension of the Cantor set $E(\Phi)$ may be obtained by following similar steps and calculations to those for finding the Fourier dimension of the Cantor set $E$; we  omit the details.

\medskip

{\noindent \bf  Acknowledgements}. This work was supported by NSFC Nos. 12171172, 12201476.


\begin{thebibliography}{100}

\bibitem{ABS22} A. Algom, S. Baker and P. Shmerkin.
\textit{On normal numbers and self-similar measures.}
Adv. Math. 399 (2022), Paper No. 108276, 17 pp.


\bibitem{ARS22}A. Algom, F.R. Hertz and Z.R. Wang.
\textit{Pointwise normality and Fourier decay for self-conformal measures.}
Adv. Math. 393 (2021), Paper No. 108096, 72 pp.


\bibitem{AR23}D. Allen and F. Ramírez.
\textit{Independence inheritance and Diophantine approximation for systems of linear forms.}
Int. Math. Res. Not. IMRN (2023), no. 2, 1760--1794.


\bibitem{ABH23} C. Aistleitner, B. Borda and M. Hauke.
\textit{On the metric theory of approximations by reduced fractions: a quantitative Koukoulopoulos-Maynard theorem.}
Compos. Math. 159 (2023),  no. 2, 207--231.


\bibitem{B86}R.C. Baker.
\textit{Diophantine inequalities.}
London Mathematical Society Monographs. New Series, 1. Oxford Science Publications. The Clarendon Press, Oxford University Press, New York, 1986.


\bibitem{BBH20} A. Bakhtawar, P. Bos and M. Hussain.
\textit{The sets of Dirichlet non-improvable numbers versus well-approximable numbers.}
Ergodic Theory Dynam. Systems 40 (2020), no. 12, 3217--3235.


\bibitem{BV06}V. Beresnevich and S. Velani.
\textit{A mass transference principle and the Duffin-Schaeffer conjecture for Hausdorff measures.}
Ann. of Math. (2) 164 (2006), no. 3, 971--992.


\bibitem{B98}C. Bluhm.
\textit{On a theorem of Kaufman: Cantor-type construction of linear fractal Salem sets.}
Ark. Mat. 36 (1998), no. 2, 307--316.


\bibitem{B09} E. Borel.
\textit{Les probabilit\'{e}s denombrables et leurs applications arithm\'{e}tiques.}
Rend. Circ. Mat. Palermo 27 (1909), 247--271.



\bibitem{B12} Y. Bugeaud.
\textit{Distribution Modulo One and Diophantine Approximation.}
 Cambridge Tracts in Mathematics 193. Cambridge: Cambridge University Press, 2012.


\bibitem{CH24} T. Cai and K. Hambrook.
\textit{On the Exact Fourier Dimension of Sets of Well-Approximable Matrices.}
(2024), arXiv: 2403.19410



\bibitem{C59} J.W.S. Cassels.
\textit{On a problem of Steinhaus about normal numbers.}
Colloq. Math. 7 (1959), 95--101.


\bibitem{CS17} X.H. Chen and A. Seeger.
\textit{Convolution powers of Salem measures with applications.}
Canad. J. Math. 69 (2017), no. 2, 284--320.


\bibitem{CZZ23} S. Chow, A. Zafeiropoulos and E. Zorin.
\textit{Inhomogeneous Kaufman Measures and Diophantine Approximation.}
(2023), arXiv:  2312.15455




\bibitem{DGW02} Y. Dayan, G. Arijit and B. Weiss.
\textit{Random walks on tori and normal numbers in self similar sets.}
(2020),  arXiv:  2002.00455



\bibitem{DEL63} H. Davenport, P. Erd\"{o}s and W.J. LeVeque.
\textit{On Weyl's criterion for uniform distribution.}
Michigan Math. J. 10 (1963), 311--314.


\bibitem{DS70}H. Davenport and W.M. Schmidt.
\textit{Dirichlet's theorem on diophantine approximation.} Symposia Mathematica, Vol. IV (INDAM, Rome, 1968/69), pp. 113--132, Academic Press, London-New York, 1970.


\bibitem{D09}A. Dubickas.
\textit{Powers of a rational number modulo 1 cannot lie in a small interval.}
Acta Arith. 137 (2009), no. 3, 233--239.


\bibitem{DS41}R.J. Duffin and A.C.  Schaeffer.
\textit{Khintchine's problem in metric Diophantine approximation.}
Duke Math. J. 8 (1941), 243--255.


\bibitem{E16} F. Ekstr\"{o}m.
\textit{Fourier dimension of random images.}
Ark. Mat. 54 (2016), no. 2, 455--471.


\bibitem{FLP95}L. Flatto, J.C. Lagarias and A.D. Pollington.
\textit{On the range of fractional parts $\{\xi(p/q)^n\}$.}
Acta Arith. 70 (1995), no. 2, 125--147.


\bibitem{FH23}R. Fraser and K. Hambrook.
\textit{Explicit Salem sets in $\mathbb{R}^{n}$.}
Adv. Math. 416 (2023), Paper No. 108901, 23 pp.

\bibitem{FW}R. Fraser and R. Wheeler.
\textit{Fourier dimension estimates for sets of exact approximation order$\colon$ the well-approximable case.}
Int. Math. Res. Not. IMRN (2023), no. 24, 20943--20969.


\bibitem{FW23} R. Fraser and R. Wheeler.
\textit{Fourier Dimension Estimates for Sets of Exact Approximation Order$\colon$ The Badly-Approximable Case.}
(2023), arXiv: 2309.05851



\bibitem{H17}K. Hambrook.
\textit{Explicit Salem sets in $\mathbb{R}^{2}$.}
Adv. Math. 311 (2017), 634--648.


\bibitem{H19} K. Hambrook.
\textit{Explicit Salem sets and applications to metrical Diophantine approximation.}
Trans. Amer. Math. Soc. 371 (2019), no. 6, 4353--4376.


\bibitem{HY23}K. Hambrook and H. Yu.
\textit{Non-Salem sets in metric diophantine approximation.}
Int. Math. Res. Not. IMRN (2023), no. 15, 13136--13152.




\bibitem{HW79} G.H. Hardy and E.M. Wright.
             \textit{An introduction to the theory of numbers.} Fifth edition.
             The Clarendon Press, Oxford University Press, New York, 1979.


\bibitem{HVW24} M. Hauke, S. Vazquez Saez and A. Walker.
\textit{Proving the Duffin-Schaeffer conjecture without GCD graphs.}
(2024), arXiv: 2404.15123



\bibitem{HS96}T. Hinokuma and H. Shiga.
\textit{Hausdorff dimension of sets arising in Diophantine approximation.}
Kodai Math. J. 19 (1996), no. 3, 365--377.


\bibitem{HS15} M. Hochman and P. Shmerkin.
\textit{Equidistribution from fractal measures.}
Invent. Math. 202 (2015), no. 1, 427--479.


\bibitem{HW19}L.L. Huang and J. Wu.
\textit{Uniformly non-improvable Dirichlet set via continued fractions.}
Proc. Amer. Math. Soc. 147 (2019), no. 11, 4617--4624.



\bibitem{HKWW18}M. Hussain, D. Kleinbock, N. Wadleigh and B.W. Wang.
\textit{Hausdorff measure of sets of Dirichlet non-improvable numbers.}
Mathematika 64 (2018), no. 2, 502--518.


\bibitem{J28}I. Jarn\'{i}k.
\textit{Zur metrischen Theorie der diopahantischen Approximationen.}
 Proc. Mat. Fyz. 36 (1928), 91--106.


\bibitem{J31} V. Jarn\'{i}k.
\textit{\"{U}ber die simultanen diophantischen Approximationen.}
Math. Z. 33 (1931), no. 1, 505--543.




\bibitem{K66} J.P. Kahane.
\textit{Images browniennes des ensembles parfaits.}
C. R. Acad. Sci. Paris S\'{e}r. A-B 263 (1966), A613--A615.


\bibitem{K80}R. Kaufman.
\textit{Continued fractions and Fourier transforms.}
Mathematika 27 (1980), no. 2, 262--267.

\bibitem{K81}R. Kaufman.
\textit{On the theorem of Jarn\'{i}k and Besicovitch.}
Acta Arith. 39 (1981), no. 3, 265--267.


\bibitem{K24}A. Khintchine.
\textit{Einige S\"{a}tze \"{u}ber Kettenbr\"{u}che, mit Anwendungen auf die Theorie der Diophantischen Approximationen.}
Math. Ann. 92 (1924), no. 1-2, 115--125.


\bibitem{K63} A. Khintchine.
\textit{Continued fractions.} Translated by Peter Wynn.
            P. Noordhoff, Ltd., Groningen 1963.


\bibitem{KL19}D.H. Kim and L.M.  Liao.
\textit{Dirichlet uniformly well-approximated numbers.}
Int. Math. Res. Not. IMRN 2019, no. 24, 7691--7732.


\bibitem{KK22}T. Kim and W. Kim.
\textit{Hausdorff measure of sets of Dirichlet non-improvable affine forms.}
Adv. Math. 403 (2022), Paper No. 108353, 39 pp.


\bibitem{KW18}D. Kleinbock and N. Wadleigh.
\textit{A zero-one law for improvements to Dirichlet's Theorem.}
Proc. Amer. Math. Soc. 146 (2018), no. 5, 1833--1844.


\bibitem{KW19}D. Kleinbock and N. Wadleigh.
\textit{An inhomogeneous Dirichlet theorem via shrinking targets.}
Compos. Math. 155 (2019), no. 7, 1402--1423.


\bibitem{KM20}D. Koukoulopoulos and J. Maynard.
\textit{On the Duffin-Schaeffer conjecture.}
Ann. of Math. (2) 192 (2020), no. 1, 251--307.
		

\bibitem{KMY24}D. Koukoulopoulos, J. Maynard and D.D. Yang.
\textit{An almost sharp quantitative version of the Duffin-Schaeffer conjecture.}
(2024), arXiv:  2404.14628



\bibitem{KN74} L. Kuipers and H. Niederreiter.
\textit{Uniform Distribution of Sequences.}
New York-London-Sydney: Wiley-Interscience, John Wiley $\&$ Sons, 1974.

\bibitem{LP09} L. \L aba and M. Pramanik.
\textit{Arithmetic progressions in sets of fractional dimension.}
Geom. Funct. Anal. 19 (2009), no. 2, 429--456.


\bibitem{LWX23} B.X. Li,  B.W. Wang and J. Xu.
\textit{Hausdorff dimension of Dirichlet non-improvable set versus well-approximable set.}
Ergodic Theory Dynam. Systems 43 (2023), no. 8, 2707--2731.


\bibitem{M95} P. Mattila.
\textit{Geometry of sets and measures in Euclidean spaces.}
Fractals and rectifiability (Cambridge Studies in Advanced Mathematics, 44). Cambridge University Press, Cambridge, 1995.


\bibitem{M15}P. Mattila.
\textit{Fourier analysis and Hausdorff dimension.}
Cambridge Studies in Advanced Mathematics, 150. Cambridge University Press, Cambridge, 2015.


\bibitem{P67} W. Philipp.
\textit{Some metrical theorems in number theory.}
Pacific J. Math. 20 (1967), 109--127.


\bibitem{PVZZ22} A.D. Pollington, S. Velani, A. Zafeiropoulos and E. Zorin.
\textit{Inhomogeneous Diophantine approximation on $M_0$-sets with restricted denominators.}
Int. Math. Res. Not. IMRN (2022), no. 11, 8571--8643.



\bibitem{P22}A. Py\"{o}r\"{a}l\"{a}.
\textit{The scenery flow of self-similar measures with weak separation condition.}
Ergodic Theory Dynam. Systems 42 (2022), no. 10, 3167--3190.



\bibitem{QR03}M. Queff\'{e}lec and O. Ramar\'{e}.
\textit{Fourier analysis of continued fractions with bounded special quotients.}
Enseign. Math. (2) 49 (2003), no. 3-4, 335--356.


\bibitem{S51}R. Salem.
\textit{On singular monotonic functions whose spectrum has a given Hausdorff dimension.}
Ark. Mat. 1 (1951), 353--365.


\bibitem{S60} W.M. Schmidt.
\textit{On normal numbers.}
Pacific J. Math. 10 (1960), 661--672.



\bibitem{S80} W.M. Schmidt,
\textit{ Diophantine approximation.}
Lecture Notes in Mathematics, 785. Springer, Berlin, 1980.



\bibitem{SS18} P. Shmerkin and V. Suomala.
\textit{Spatially independent martingales, intersections, and applications.}
Mem. Amer. Math. Soc. 251 (2018), no. 1195, v+102 pp.


\bibitem{S58} P. Sz\"{u}sz.
\textit{\"{U}ber die metrische Theorie der Diophantischen Approximation.}
Acta Math. Acad. Sci. Hungar. 9 (1958), 177--193.



\bibitem{V85} J.D. Vaaler.
\textit{Some extremal functions in Fourier analysis.}
Bull. Amer. Math. Soc. (N.S.) 12 (1985), no. 2, 183--216.


\bibitem{W12} M. Waldschmidt .
\textit{Recent advances in Diophantine approximation}. Number theory, analysis and
 geometry, 659--704, Springer, New York, 2012.



\bibitem{Wall50} D.D. Wall.
\textit{Normal numbers.}
Thesis (Ph.D.)-University of California, Berkeley. 1950.



\bibitem{Weyl16} H. Weyl.
\textit{\"{U}ber die Gleichverteilung von Zahlen mod. Eins.}
Math. Ann. 77 (1916), no. 3, 313--352.


\bibitem{Y19}H. Yu.
\textit{A Fourier-analytic approach to inhomogeneous Diophantine approximation.}
Acta Arith. 190 (2019), no. 3, 263--292.

\bibitem{Y21} H. Yu.
\textit{On the metric theory of inhomogeneous Diophantine approximation: an Erd\"{o}s-Vaaler type result.}
J. Number Theory 224 (2021), 243--273.


\end{thebibliography}
\end{document}